%% file: FastPartialSpectralDecomposition.tex
\documentclass{article}

\input{header}

\DeclareFontEncoding{LS1}{}{}
\DeclareFontEncoding{LS2}{}{\noaccents@}
\DeclareFontSubstitution{LS1}{stix}{m}{n}
\DeclareFontSubstitution{LS2}{stix}{m}{n}
\DeclareSymbolFont{mysymbols}{LS1}{stixscr}{m}{n}
\DeclareSymbolFont{myarrows3}{LS2}{stixtt}{m}{n}
\DeclareSymbolFont{myintegrals}{LS2}{stixcal}{m}{n}
\DeclareMathSymbol{\myvdots}{\mathrel}{mysymbols}{"34}
\DeclareMathSymbol{\mycdots}{\mathrel}{mysymbols}{"35}
\DeclareMathSymbol{\myadots}{\mathrel}{mysymbols}{"36}
\DeclareMathSymbol{\myddots}{\mathrel}{mysymbols}{"37}
\DeclareMathSymbol{\blacksquare}{\mathord}{mysymbols}{"B7}
\DeclareMathSymbol{\square}{\mathord}{mysymbols}{"B8}
\DeclareMathSymbol{\lrblacktriangle}{\mathord}{mysymbols}{"F9}
\DeclareMathSymbol{\llblacktriangle}{\mathord}{mysymbols}{"FA}
\DeclareMathSymbol{\urblacktriangle}{\mathord}{mysymbols}{"FB}
\DeclareMathSymbol{\urblacktriangle}{\mathord}{mysymbols}{"FC}
\DeclareMathSymbol{\urtriangle}{\mathord}{myarrows3}{"98}
\DeclareMathSymbol{\lltriangle}{\mathord}{myarrows3}{"99}
\DeclareMathSymbol{\mydiagdown}{\mathord}{myintegrals}{"3C}

\graphicspath{{./Codes/}}

\begin{document}

\title{On symmetrizing the ultraspherical spectral method for self-adjoint problems}

\author{%
{\sc Jared Lee Aurentz\thanks{Email: Jared.Aurentz@icmat.es}~~and Richard Mika\"el Slevinsky\thanks{Email: Richard.Slevinsky@umanitoba.ca}}\\[2pt]
Instituto de Ciencias Matem\'aticas, 28049, Madrid, Spain\\
Department of Mathematics, University of Manitoba, Winnipeg, Canada}

\maketitle

\begin{abstract}
A mechanism is described to symmetrize the ultraspherical spectral method for self-adjoint problems. The resulting discretizations are symmetric and banded. An algorithm is presented for an adaptive spectral decomposition of self-adjoint operators. Several applications are explored to demonstrate the properties of the symmetrizer and the adaptive spectral decomposition.

\end{abstract}

\section{Introduction}

Let $H$ be a separable Hilbert space with inner product $\langle\cdot,\cdot\rangle$ and norm $\norm{\cdot}$. We are interested in the solution of the self-adjoint linear differential eigenvalue problem with self-adjoint boundary conditions posed on a subset of the real line, $D\subset\R$. The $2N^{\rm th}$ order problem in self-adjoint form reads:
\begin{equation}\label{eq:EVP}
\LL u := \left[(-\DD)^N(p_N\DD^N) + (-\DD)^{N-1}(p_{N-1}\DD^{N-1}) + \cdots + p_0\right]u = \lambda w u,\qquad \BB u = 0,
\end{equation}
where the usual conditions on the coefficients $p_i$ and $w$ are required for sensible solution; at the very least, $p_N\ne 0$, and $w\ge 0$ and $w = 0$ only on a set of Lebesgue measure zero. In this form, $\LL$ is a self-adjoint linear differential operator in $H = L^2(D,w(x)\ud x)$ provided it is accompanied by $\BB : H \to \C^{2N}$, a set of $2N$ self-adjoint boundary conditions. We will also assume that $p_i\in\P_m$, that is, the variable coefficients are at most degree-$m$ polynomials. The self-adjoint linear differential eigenvalue problem is an important problem in applied mathematics, with many scientific contributions furthering our collective understanding, too many to all be cited.

Given the quotient Hilbert space $H_\BB := \{ u \in H : \BB u = 0\}$, a classical approach calls for a Ritz--Galerkin discretization for the trial and test functions taken from an $n$-dimensional subspace $V_n\subset H_\BB$. Indeed, when $V_n = \P_{n-1}$ and an orthonormal polynomial basis satisfying the boundary conditions is utilized, the discrete realization of $\langle u, \LL v\rangle$ is automatically symmetric (hermitian). Unfortunately, the resulting discretizations are dense. Thus data-sparsity in problem formulations, such as Eq.~\eqref{eq:EVP} with finite-degree polynomial coefficients, is ignored when transformed into a computational problem, resulting in $\OO(n^3)$ complexity algorithms~\cite{Golub-Van-Loan-13} for approximate eigenvalues and eigenfunctions.

Another classical approach is a pseudospectral method implemented by collocation~\cite{Canuto-23-815-86,Driscoll-229-5980-10}. Apart from the boundary conditions, differentiation matrices resulting from spectral collocation have many desirable properties such as (skew) centrosymmetry. The traditional wisdom is that, symmetry notwithstanding, even-ordered spectral differentiation matrices have real and distinct eigenvalues~\cite{Gottlieb-Lustman-20-909-83}. One of the main draws to collocation is the ease of implementation; variable coefficients are implemented as diagonal scalings. Some formulations are even symmetric for Sturm--Liouville problems~\cite{Chen-Ma-206-748-08}. Unfortunately, the discretizations are also dense and join the Ritz--Galerkin discretization behind the $\OO(n^3)$ computational barrier.

A contemporary approach to the solution of linear differential equations with polynomial coefficients is the Olver--Townsend ultraspherical spectral method~\cite{Olver-Townsend-55-462-13}. Compared with the finite-dimensional classical schemes, it is an infinite-dimensional Petrov--Galerkin scheme that utilizes multiple ultraspherical polynomial bases to represent a linear differential operator such as the left-hand side of Eq.~\eqref{eq:EVP} as sparsely as possible. This spectral method is based on the facts that: differentiation of Chebyshev polynomials of the first kind is diagonal when represented in terms of ultraspherical polynomials $C_n^{(\lambda)}$ of parameter $\lambda=1$; subsequent differentiations are also diagonal when represented in higher order ultraspherical polynomial bases; conversion from a representation in low order ultraspherical polynomials to higher order ones is banded; and, polynomial coefficient multiplication is banded. Crucially, the bandedness implies that for linear differential equations of the form:
\begin{equation}
\LL u = f,\quad{\rm with}\quad \BB u = c,
\end{equation}
representations of the solution are achieved in $\OO(m^2 n)$ complexity, and the degree $n$ may be determined adaptively through a running estimate of the forward error without restarting the partial $QR$ factorization of $\LL$. When posing Eq.~\eqref{eq:EVP} as an eigenvalue problem, one approach in adaptations of the ultraspherical spectral method to polar~\cite{Vasil-et-al-325-53-16} and spherical~\cite{Lecoanet-et-al-1804-09283} geometries is to impose:
\begin{equation}
\begin{pmatrix} \BB\\ \LL\end{pmatrix} u = \lambda\begin{pmatrix} 0\\ \CC\end{pmatrix} u,
\end{equation}
where $\CC$ is the operator that would convert the representation of the solution in the ultraspherical polynomial basis of the domain of $\LL$ to the basis for its range. Unfortunately, even for a self-adjoint linear differential eigenvalue problem with self-adjoint boundary conditions, symmetry appears to be sacrificed for the banded sparsity\footnote{In exceptional circumstances, such as when $\BB$ is zero apart from the first few columns and the nonzero part prepended to $\LL$ renders the block operator discretization self-adjoint, symmetry may still be within reach. Many common boundary conditions, however, are genuinely infinite-dimensional. This destruction of symmetry has real consequences for the spectra of finite truncations. Whereas Chebyshev and Legendre second-order differentiation matrices by collocation have real spectra~\cite{Weideman-Trefethen-25-1279-88}, it is easy to show by counter-example that finite truncations of ultraspherical discretizations of self-adjoint operators may produce complex spectra.}. 

In this work, we describe a procedure to formulate self-adjoint problems such as Eq.~\eqref{eq:EVP} in order that symmetry may transcend the discretization and that data-sparsity in the problem is preserved in the form of a banded discretization. Thus, we describe a procedure to sparsify the Ritz--Galerkin scheme or to symmetrize the ultraspherical spectral method. With symmetric-definite banded spectral discretizations, $\OO(m n^2)$ complexity partial diagonalization is possible through finite-dimensional truncation and numerical linear algebra. We also present an adaptive infinite-dimensional spectral decomposition of self-adjoint linear differential operators that deflates resolved eigenpairs on-the-fly and for certain problems, achieves linear complexity partial spectral decompositions. There are many advantages of preserving symmetry. Symmetry ensures that the general Rayleigh quotient iteration~\cite{Rayleigh-37,Parlett-28-679-74} converges cubically rather than quadratically. And for time-stepping schemes, symmetry precludes rogue eigenvalues from being scattered spuriously in the complex plane. There are also advantages of preserving the data-sparsity via a small bandwidth. This enables spectral analysis of infinite-dimensional discretizations via, e.g. Gerschgorin disks~\cite{Golub-Van-Loan-13} and its generalizations~\cite{Nakatsukasa-80-2127-11}. One disadvantage of the ultraspherical spectral method is that it applies to the strong formulation of the problem, preventing the resolution of weak eigenfunctions unless they are as smooth as the strong ones.

There are a handful of spectral schemes that lead to symmetric-definite and banded discretizations in the literature. These include the Fourier spectral method with periodic boundary conditions~\cite{Boyd-00} and the Hermite function and Tanh--Chebyshev spectral methods~\cite{Iserles-Webb-NA2019/01} as two examples of many orthogonal systems on $L^2(\R)$. There are also examples with nontrivial boundary conditions as well. In~\cite{Shen-15-1489-94}, Shen finds symmetric and banded spectral Legendre-based discretizations of constant coefficient second- and fourth-order problems. He repeats the procedure with Chebyshev-based discretizations~\cite{Shen-16-74-95} and while the efficiency is comparable, the symmetry is lost. Banded symmetry is also common in the finite element method~\cite{Beuchler-Schoberl-103-339-06}. Through examples, we will show precisely when we should anticipate classical and other orthogonal polynomial-based discretizations for self-adjoint problems lead to self-adjoint and banded discretizations, even on non-compact domains and with piecewise-defined bases.

Polynomial approximation theory suggests that we should seek to numerically represent the variable coefficients in Eq.~\eqref{eq:EVP} as Chebyshev polynomial expansions~\cite{Trefethen-12}. An adaptation of the fast multipole method (FMM)~\cite{Greengard-Rokhlin-73-325-87} accelerates the Chebyshev--Legendre~\cite{Alpert-Rokhlin-12-158-91}, ultraspherical--ultraspherical~\cite{Keiner-31-2151-09}, and Jacobi--Jacobi~\cite[\S 3.3]{Keiner-11} connection problems to linear complexity in the degree. Provided the Jacobi parameters are not too large, the acceleration of the connection problem enables Jacobi polynomial expansions with nearly the same rapidity as Chebyshev polynomial expansions. Thus in the problems in this paper, the variable coefficients in Eq.~\eqref{eq:EVP} are in fact polynomials, but in practice they could very well be good numerical approximations to functions in a particular space. How such approximations, in particular of the leading coefficient $p_N$, affect the spectrum is beyond the scope of this report.

\section{A well-known model problem}\label{section:modelproblem}

It is instructive to begin with a well-known model problem for illustrative purposes.

Consider the self-adjoint Sturm--Liouville problem~\cite{Amrein-Hinz-Pearson-05}:
\begin{equation}
-\DD^2 u = \lambda u, \qquad u(\pm1) = 0.
\end{equation}
Suppose we represent eigenfunctions in terms of the weighted Jacobi polynomial expansion:
\begin{equation}
u(x) = \sum_{n=0}^\infty u_n (1-x^2)\tilde{P}_n^{(1,1)}(x).
\end{equation}
Here, we use the tilde to denote orthonormalization of the orthogonal polynomials with respect to their orthogonality measure. In particular, the normalized Jacobi polynomials are related to the standard normalization through~\cite[\S 18.3]{Olver-et-al-NIST-10}:
\begin{equation}
\tilde{P}_n^{(\alpha,\beta)}(x) := \sqrt{\dfrac{(2n+\alpha+\beta+1)\Gamma(n+\alpha+\beta+1)n!}{2^{\alpha+\beta+1}\Gamma(n+\alpha+1)\Gamma(n+\beta+1)}}P_n^{(\alpha,\beta)}(x).
\end{equation}
It is clear that our representation satisfies the Dirichlet boundary conditions. Furthermore, differentiation follows from:
\begin{equation}
\DD\begin{pmatrix} (1-x^2)\tilde{P}_0^{(1,1)}(x) & (1-x^2)\tilde{P}_1^{(1,1)}(x) & \cdots \end{pmatrix} = \begin{pmatrix} \tilde{P}_0^{(0,0)}(x) & \tilde{P}_1^{(0,0)}(x) & \cdots \end{pmatrix}\begin{pmatrix} 0 & -\sqrt{2}\\ & & -\sqrt{6}\\ & & & -\sqrt{12}\\ & & & & \ddots\end{pmatrix},
\end{equation}
and (negative) second differentiation from:
\begin{equation}
-\DD\begin{pmatrix} \tilde{P}_0^{(0,0)}(x) & \tilde{P}_1^{(0,0)}(x) & \cdots \end{pmatrix} = \begin{pmatrix} \tilde{P}_0^{(1,1)}(x) & \tilde{P}_1^{(1,1)}(x) & \cdots \end{pmatrix}\begin{pmatrix} 0\\ -\sqrt{2}\\ & -\sqrt{6}\\ & & -\sqrt{12}\\ & & & \ddots\end{pmatrix}.
\end{equation}
Ultimately, our representation of the negative second derivative is diagonal, when interpreted in the appropriate bases:
\begin{equation}
-\DD^2 = \begin{pmatrix} d_0\\ & d_1\\ & & d_2\\ & & & \ddots\end{pmatrix},\quad{\rm where}\quad d_n = (n+1)(n+2).
\end{equation}
A representation of the proposed eigenfunction in the basis of the expansion of $-\DD^2 u$ can be realized through the symmetric positive-definite pentadiagonal multiplication operator:
\begin{equation}
\MM[1-x^2] = \begin{pmatrix} a_0 & 0 & b_0\\ 0 & a_1 & & \ddots\\ b_0 & & a_2\\ & \ddots & & \ddots\end{pmatrix},\quad{\rm where}\quad a_n = \frac{2(n+1)(n+2)}{(2n+1)(2n+5)},\quad{\rm and}\quad b_n = -\sqrt{\frac{(n+1)(n+2)(n+3)(n+4)}{(2n+3)(2n+5)^2(2n+7)}}.
\end{equation}
In this problem, a symmetric-definite banded Petrov--Galerkin discretization is nearly immediate. The key element is choosing an appropriate polynomial basis for $H_\BB$, and the discretization unfolds naturally.

This problem is well-known and difficult to generalize with the transcendent symmetry because the Dirichlet conditions are so easily enforced by weighting a reasonable polynomial basis by $1-x^2$ that the underlying mechanisms are hidden. To understand the mechanisms at work, we will deconstruct the discretization in abstract terms.

\section{Polynomial bases satisfying the boundary conditions}

Let $\{\phi_n\}_{n=0}^\infty$ be an orthonormal polynomial basis for $H$ with $\deg(\phi_n) = n$. Basis recombination~\cite{Shen-15-1489-94,Shen-16-74-95,Doha-Abd-Elhameed-24-548-02,Doha-Bhrawy-42-137-06,Julien-Watson-228-1480-09,Livermore-229-2046-10,Doha-Bhrawy-64-558-12} is the idea of recombining a usually contiguous subset of a basis to satisfy additional constraints.

For linearly independent boundary conditions in $\BB$, let $\{\rho_n\}_{n=0}^\infty$ be the recombinations of the orthonormal basis $\{\phi_n\}_{n=0}^\infty$ that annihilate $\BB$, guaranteeing $\BB\rho_n = 0$. Note that $\deg(\rho_n)$ is not necessarily $n$ as in the previous problem, but we will assume that $\rho_n$ are nevertheless degree-graded in the generalized sense that there exists a positive integer $\nu$ such that $\deg(\rho_{n+\nu}) > \deg(\rho_n)$ is always true. This assumption suggests the existence of an infinite-dimensional lower-triangular and banded conversion such that:
\begin{equation}\label{eq:rhobyphi}
\begin{pmatrix} \rho_0 & \rho_1 & \cdots \end{pmatrix} = \begin{pmatrix} \phi_0 & \phi_1 & \cdots \end{pmatrix} A.
\end{equation}
With the conversion in hand, a banded $QR$ factorization reveals the isometry mapping the orthonormal polynomial basis for $H$ to the orthonormal polynomial basis $\{\psi_n\}_{n=0}^\infty$ for $H_\BB$:
\begin{equation}\label{eq:psibyphi}
\begin{pmatrix} \psi_0 & \psi_1 & \cdots \end{pmatrix} = \begin{pmatrix} \phi_0 & \phi_1 & \cdots \end{pmatrix} Q,
\end{equation}
and the auxiliary resulting conversion:
\begin{equation}\label{eq:rhobypsi}
\begin{pmatrix} \rho_0 & \rho_1 & \cdots \end{pmatrix} = \begin{pmatrix} \psi_0 & \psi_1 & \cdots \end{pmatrix} R.
\end{equation}
We note that a banded $QR$ factorization of $A\in\C^{(n+2N)\times n}$ requires $\OO(n)$ operations and either factor may be applied to a vector with the same cost. This shows that all three of the aforementioned conversions run in linear complexity without the effort of a dense Gram--Schmidt procedure.

Finally, define the polynomials $\{\sigma_n\}_{n=0}^\infty$ by:
\begin{equation}\label{eq:psibysigma}
\begin{pmatrix} \sigma_0 & \sigma_1 & \cdots \end{pmatrix}R^* = \begin{pmatrix} \psi_0 & \psi_1 & \cdots \end{pmatrix}.
\end{equation}
And since:
\begin{equation}\label{eq:phibypsi}
\begin{pmatrix} \psi_0 & \psi_1 & \cdots \end{pmatrix}Q^* = \begin{pmatrix} \phi_0 & \phi_1 & \cdots \end{pmatrix},
\end{equation}
we gather all three adjoint conversions:
\begin{equation}\label{eq:phibysigma}
\begin{pmatrix} \sigma_0 & \sigma_1 & \cdots \end{pmatrix}A^* = \begin{pmatrix} \phi_0 & \phi_1 & \cdots \end{pmatrix}.
\end{equation}
In order to keep the discussion concrete, we relate the so-defined polynomials to the simple canonical problem of \S~\ref{section:modelproblem}; we identify the orthonormal polynomial basis for $H = L^2([-1,1])$ as the normalized Legendre polynomials, $\phi_n(x) = \tilde{P}_n(x)$. The recombination to satisfy the Dirichlet boundary conditions are in fact weighted Jacobi polynomials:
\begin{equation}
\rho_n(x) = \sqrt{\frac{(n+1)(n+2)}{(2n+1)(2n+3)}}\tilde{P}_n(x) - \sqrt{\frac{(n+1)(n+2)}{(2n+3)(2n+5)}}\tilde{P}_{n+2}(x) = (1-x^2)\tilde{P}_n^{(1,1)}(x),
\end{equation}
by~\cite[\S 18.7.2, 18.9.8]{Olver-et-al-NIST-10}, and furthermore an orthonormal basis for $H_\BB$ is $\psi_n(x) = (1-x^2)\tilde{P}_n^{(2,2)}(x)$. We may also identify the auxiliary polynomials $\sigma_n(x) = \tilde{P}_n^{(1,1)}(x)$, though this identification is serendipitous rather than necessary, and in the general scenario these polynomials are almost never identifiable as classical orthogonal polynomials.

To simplify the notation, we will use a shorthand for an expansion in a basis and a vector of coefficients:
\begin{equation}
u(x) = \sum_{n=0}^\infty \psi_n(x) u_n = \begin{pmatrix} \psi_0 & \psi_1 & \cdots\end{pmatrix}\begin{pmatrix} u_0\\u_1\\\vdots\end{pmatrix} = \bs{\psi}^\top \bs{u}.
\end{equation}

We will now prove that the Petrov--Galerkin scheme for Eq.~\eqref{eq:EVP} is symmetric-definite and banded, where the solution is represented by $u(x) = \bs{\rho}^\top \bs{v}$ and the residuals are estimated using $\bs{\sigma}$.

\subsection{The Petrov--Galerkin scheme is self-adjoint}

Having identified that $\bs{\psi}$ form an orthonormal polynomial basis for $H_\BB$, it follows that the Ritz--Galerkin scheme is self-adjoint:
\begin{equation}
\LL \bs{\psi}^\top \bs{u} = \bs{\psi}^\top L\bs{u} = \lambda w\bs{\psi}^\top \bs{u} = \lambda \bs{\psi}^\top M\bs{u},
\end{equation}
where $L = L^*\in\C^{\infty\times\infty}$ and $M = M^*\in\C^{\infty\times\infty}$ is positive-definite. We let $\bs{u} = R\bs{v}$ and we use Eq.~\eqref{eq:psibysigma} to represent $\bs{\psi}$ in terms of $\bs{\sigma}$. Then we arrive at the Petrov--Galerkin scheme:
\begin{equation}
\bs{\sigma}^\top R^*LR \bs{v} = \lambda \bs{\sigma}^\top R^*MR\bs{v}.
\end{equation}
Clearly, both $R^*LR$ and $R^*MR$ are self-adjoint and the latter is also positive-definite.

\subsection{The Petrov--Galerkin scheme is banded}

Let us consider the Petrov--Galerkin scheme:
\begin{equation}
\LL \bs{\rho}^\top \bs{v} = \lambda w \bs{\rho}^\top \bs{v}.
\end{equation}
Combining Eqs.~\eqref{eq:rhobyphi} and~\eqref{eq:phibysigma}, we know that $w\bs{\rho}^\top = w\bs{\phi}^\top A = \bs{\phi}^\top M_B A = \bs{\sigma}^\top A^*M_B A$, where $M_B = M_B^*$ is also positive-definite and banded by executing the Clenshaw--Smith algorithm with the orthonormal polynomials' Jacobi operator~\cite[\S 4.1]{Slevinsky-Olver-332-290-17}. Similarly, $\LL \bs{\rho}^\top = \LL \bs{\phi}^\top A$, and by the ultraspherical spectral method~\cite{Olver-Townsend-55-462-13}, there exists a banded discretization of $\LL$, namely $L_B$, mapping (normalized) Legendre polynomials to (normalized) Jacobi polynomials with integer parameters in proportion to the highest derivatives. Applying the inverse (banded) conversion operator, and further utilizing $\bs{\phi}^\top = \bs{\sigma}^\top A^*$, we arrive at:
\begin{equation}\label{eq:PetrovGalerkinisBanded}
\bs{\sigma}^\top A^* C^{-1} L_B A \bs{v} = \lambda \bs{\sigma}^\top A^*M_BA\bs{v}.
\end{equation}
At first glance, this formula does not appear to guarantee a banded discretization, for only in special circumstances are the inverses of banded matrices banded themselves~\cite{Strang-107-12413-10}. To show that it is indeed banded, we note that the conversion matrix $C$ is upper-triangular and banded, which means that its inverse will {\em not} extend the bandwidth of $L_BA$ any lower. Since $A^*$ is also upper-triangular and since the Petrov--Galerkin scheme is self-adjoint, we conclude that it must be banded above as well.

\begin{remark}
\begin{enumerate}
\item This proof does not imply that $C^{-1} L_B$ is self-adjoint; this operator may be strictly upper-triangular as in the model problem! But the composition with $A^*$ and $A$ on either side ensures the final scheme is self-adjoint.
\item The choice of normalization of the polynomial basis that satisfies the boundary conditions amounts to a congruence transformation, therefore does not affect the symmetry. On the other hand, the operator $C^{-1} L_B$ must be posed on the orthonormal basis rather than any other normalization.
\end{enumerate}
\end{remark}

\subsection{Analysis of the conversions $A$ and $R$}

The orthogonal polynomial basis for $H_\BB$ has already been described by Livermore~\cite{Livermore-229-2046-10}, who calls them Galerkin orthogonal polynomials. Livermore presents analytic and asymptotic results for the polynomials to suggest that they are reasonable bases for the constrained Hilbert space, as they mimic many of the properties of the original orthogonal polynomials for $H$.

Here, our analysis of the basis recombination is more abstract.

\begin{definition}
Let $\ell_\lambda^2\subset\C^\infty$ be the discrete Sobolev space equipped with the norm:
\begin{equation}
\norm{\bs{u}}_{\ell_\lambda^2}^2 = \sum_{n=0}^\infty \abs{u_n}^2(n+1)^{2\lambda}.
\end{equation}
When $\lambda = 0$, it is convenient to drop it from the notation.
\end{definition}
Colloquially, we will refer to the conversion matrices $A$ as {\em annihilators} in relation to $\BB A = 0$. For infinite-dimensional annihilators $A\in\C^{\infty\times\infty}$ from Eq.~\eqref{eq:rhobyphi}, we wish to show that the induced $\ell^2$-norm:
\begin{equation}
\norm{A}_{\ell^2\to\ell^2} := \sup_{\norm{\bs{x}}_{\ell^2}\le 1} \norm{A\bs{x}}_{\ell^2},
\end{equation}
is bounded. For if $A:\ell^2\to\ell^2$ is a bounded operator, then the new polynomials $\{\rho_n\}$ form a basis for the constrained Hilbert space $H_\BB$, and furthermore the $QR$ factorization exists in the strong operator topology~\cite[Theorem 31]{Hansen-254-2092-08}. This argument does not prove that the annihilator is numerically well-conditioned, as this also depends on a bound for the norm of the inverse of the principal finite-dimensional section of $A$.

In general:
\begin{equation}
\BB =\begin{pmatrix}
b_{0,0} & b_{0,1} & b_{0,2} & \cdots\\
b_{1,0} & b_{1,1} & b_{1,2} & \cdots\\
\vdots & \vdots & \vdots & \cdots\\
b_{2N-1,0} & b_{2N-1,1} & b_{2N-1,2} & \cdots\\
\end{pmatrix}\in \C^{2N\times\infty},
\end{equation}
and the matrix $A$ is uniquely determined if it is unit lower-triangular and banded and if it annihilates $\BB$. Thus the nontrivial part of the $k^{\rm th}$ column of $A$ must satisfy:
\begin{equation}\label{eq:LinearSystemForA}
\begin{pmatrix}
b_{0,k+1} & b_{0,k+2} & \cdots & b_{0,k+2N}\\
b_{1,k+1} & b_{1,k+2} & \cdots & b_{1,k+2N}\\
\vdots & \vdots & \ddots & \vdots\\
b_{2N-1,k+1} & b_{2N-1,k+2} & \cdots & b_{2N-1,k+2N}
\end{pmatrix}
\begin{pmatrix}
a_{k+1,k}\\
a_{k+2,k}\\
\vdots\\
a_{k+2N,k}
\end{pmatrix}
=
-\begin{pmatrix}
b_{0,k}\\
b_{1,k}\\
\vdots\\
b_{2N-1,k}
\end{pmatrix}.
\end{equation}
A proof of the existence and uniqueness of a solution to Eq.~\eqref{eq:LinearSystemForA} as well as an upper bound on $\norm{a_{:,k}}_{\ell^2}$ would require a bound on the norm of the inverse of every matrix defined by Eq.~\eqref{eq:LinearSystemForA}. Below, we prove a simple result for practical applications.
\begin{theorem}
Let $\BB\in\R^{2\times\infty}$ be a pair of real linear functionals such that:
\begin{align}
c_{i,0}(n+1)^{\lambda_i} \le \abs{b_{i,n}} \le c_{i,1}(n+1)^{\lambda_i},\qquad i = 0,1,
\end{align}
where $0\le \lambda_0,\lambda_1<\infty$ and $0<c_{0,0},c_{0,1},c_{1,0},c_{1,1}<\infty$, and furthermore assume that the coefficients of one functional are positive and the other alternate in sign. Then $\norm{A}_{\ell^2\to\ell^2} < \infty$ and the infinite-dimensional $QR$ factorization exists.
\end{theorem}
\begin{proof}
The solution of the linear system in Eq.~\eqref{eq:LinearSystemForA} is:
\begin{equation}
\begin{pmatrix}
a_{k+1,k}\\~\\
a_{k+2,k}
\end{pmatrix}
=
\begin{pmatrix}
\dfrac{b_{0,k+2}b_{1,k} - b_{1,k+2}b_{0,k}}{b_{0,k+1}b_{1,k+2}-b_{0,k+2}b_{1,k+1}}\\~\\
\dfrac{b_{1,k+1}b_{0,k} - b_{0,k+1}b_{1,k}}{b_{0,k+1}b_{1,k+2}-b_{0,k+2}b_{1,k+1}}
\end{pmatrix}.
\end{equation}
As one of the functionals is alternating in sign:
\begin{align}
\abs{b_{0,k+1}b_{1,k+2}-b_{0,k+2}b_{1,k+1}} & = \abs{b_{0,k+1}b_{1,k+2}}+\abs{b_{0,k+2}b_{1,k+1}},\\
& \ge c_{0,0}c_{1,0}\left[(k+2)^{\lambda_0}(k+3)^{\lambda_1}+(k+3)^{\lambda_0}(k+2)^{\lambda_1}\right],
\end{align}
and it follows that for every $\lambda_0$ and $\lambda_1$, both sub-diagonal entries are absolutely bounded, independent of $k$. By an estimate on the singular values of a matrix~\cite{Qi-56-105-84}:
\begin{equation}
\norm{A}_{\ell^2\to\ell^2} \le \sup_{k\in\N_0}\max \left(1+\abs{a_{k+1,k}}+\abs{a_{k+2,k}}, \abs{a_{k+2,k}}+\abs{a_{k+2,k+1}}+1\right) < \infty.
\end{equation}
Since $A$ is bounded, the $QR$ factorization exists and $\norm{A}_{\ell^2\to\ell^2} = \norm{R}_{\ell^2\to\ell^2} < \infty$.
\end{proof}

For concreteness, consider the case of separated boundary conditions on the interval $[-1,1]$. Legendre polynomials are positive at the right endpoint and alternate in sign at the other~\cite[\S 18.6]{Olver-et-al-NIST-10}. Since they and their derivatives at the endpoints are bounded above and below by algebraic powers of the degree, a pair of separated boundary conditions of the form $a_\pm u(\pm1) + b_\pm u'(\pm1) = 0$, with the conditions $a_+ b_+ > 0$ and $a_-b_- < 0$, produce bounded conversion operators $A$ and $R$.

An alternative to contiguous basis recombination of the form $\phi_n + a_{n+1,n}\phi_{n+1} + a_{n+2,n}\phi_{n+2}$ is the recombination with only the first two elements, $\phi_n + \hat{a}_{0,n}\phi_0 + \hat{a}_{1,n}\phi_1$, which has also been discussed elsewhere~\cite{Shen-15-1489-94,Livermore-229-2046-10}. Not only are such recombinations not banded, but they are also not bounded in the above scenario, which demonstrates the superiority of the contiguous approach.

\subsection{Implementation details}

The solution of the linear system in Eq.~\eqref{eq:LinearSystemForA} is performed by a pivoted $LU$ factorization. The number of boundary conditions is usually quite small, say $N\le10$, so that instability is rarely an issue in practice. Although we anticipate algebraic growth in the functionals that realize the boundary conditions, if $b_{2N-1,n} \le (n+1)^{2N}$, then overflow in IEEE $64$-bit floating-point arithmetic only begins after $n \ge 2^{\frac{512}{N}}$, which is rather large for spectral methods on modern computers. This method is referred to as the standard method.

At the very least, self-adjoint boundary conditions $\BB$ must be linearly independent. But linear independence in the infinite-dimensional sense does not imply that each and every linear system in Eq.~\eqref{eq:LinearSystemForA} is soluble. There are a few pathological combinations of boundary conditions and bases that require careful attention. To not overburden the reader, we describe the implementation details for boundary conditions on orthogonal polynomials with the standard normalization.

\subsubsection{The basis satisfies some of the boundary conditions}

One polynomial basis for $L^2([-1,1])$ consists of $\{1,x\}\oplus\{P_n(x)-P_{n-2}(x)\}_{n=2}^\infty$, and it is not unreasonable to use this basis to satisfy Dirichlet and Neumann boundary conditions for a fourth-order eigenvalue problem, $u(\pm1) = u'(\pm1) = 0$. In fact, every element of the basis for $n\ge2$ already satisfies the Dirichlet boundary conditions. The boundary conditions are:
\begin{equation}
\BB = \begin{pmatrix} 1 & -1 & 0 & 0 & 0 & 0 & 0 & \cdots\\ 1 & 1 & 0 & 0 & 0 & 0 & 0 & \cdots\\ 0 & 1 & -3 & 5 & -7 & 9 & -11 & \cdots\\ 0 & 1 & 3 & 5 & 7 & 9 & 11 & \cdots\end{pmatrix}.
\end{equation}
Most of this $\BB$'s finite sections in use in Eq.~\eqref{eq:LinearSystemForA} are row-rank deficient.

\subsubsection{A column of zeros appears}

For a second-order problem, consider the self-adjoint and separated boundary conditions:
\begin{equation}\label{eq:ZeroColumnBoundaryConditions}
\nu(\nu+1) u(1) - 2u'(1) = \nu(\nu+1) u(-1) + 2u'(-1) = 0.
\end{equation}
If we pose these boundary conditions on the Legendre polynomials, then a column of zeros appears precisely when $\nu = n$. This is because the degree-$n$ Legendre polynomial already satisfies the boundary conditions. In fact, analytical formulas for the Legendre recombination of the Robin boundary conditions, $a_\pm u(\pm1) + b_\pm u'(\pm1) = 0$, are presented by Doha and Bhrawy~\cite[\S 2.1]{Doha-Bhrawy-64-558-12}:
\begin{equation}
\rho_n(x) = P_n(x) + \zeta_n(a_\pm,b_\pm) P_{n+1}(x) + \eta_n(a_\pm, b_\pm) P_{n+2}(x),
\end{equation}
where:
\begin{align}
\zeta_n(a_\pm,b_\pm) & = -\frac{2(2n+3)(a_+b_-+b_+a_-)}{2((n+2)^2b_--2a_-)+(n+2)^2b_+((n+1)(n+3)b_--2a_-)},\quad{\rm and}\\
\eta_n(a_\pm,b_\pm) & = -\frac{2a_+((n+1)^2b_--2a_-)+(n+1)^2b_+(n(n+2)b_--2a_-)}{2((n+2)^2b_--2a_-)+(n+2)^2b_+((n+1)(n+3)b_--2a_-)}.
\end{align}
We stress that this analytical method fails in the neighbourhood of $n = \nu$, for the reasons described above. Surely, $\BB P_\nu \equiv 0,$ but this creates a column of zeros in the linear functionals, rendering the adjacent linear systems column-rank deficient. Rank-deficient boundary conditions were also overlooked for ultraspherical and Jacobi polynomial recombinations~\cite[\S 4.1]{Doha-Abd-Elhameed-24-548-02,Doha-Bhrawy-42-137-06}.

Both of these pathological scenarios are soluble by doubling the lower bandwidth of $A$. This leads to our so-called pathological method, which returns a minimum $2$-norm solution for the $k^{\rm th}$ column of $A$ using the matrix $\BB_{0:2N-1, k+1:k+4N}$ and the same right-hand side as in the standard method. As the systems may be rank deficient, we solve them using a Moore--Penrose pseudoinverse~\cite{Golub-Van-Loan-13} rather than a complete orthogonal factorization. We note that both of these scenarios are pathological but almost never appear. This can be argued because there are a countably infinite subset of boundary conditions, with Lebesgue measure zero, such as those described in Eq.~\eqref{eq:ZeroColumnBoundaryConditions} of an uncountably infinite set in general, since $a_\pm$ and $b_\pm$ range over the reals. 

Although we do not have a formal proof that doubling the lower bandwidth ensures the existence of $A$ for $2N^{\rm th}$ order problems, our rationale is that the $2N$ boundary conditions are linear combinations of the $4N$ endpoint conditions $u^{(\nu)}(\pm1)$, where $\nu=0,\ldots,2N-1$, that are linearly independent in the infinite-dimensional sense. And although most boundary conditions require only $2N$ lower bands, some may require $4N$. It is not inconceivable, however, that there may be other pathological boundary conditions not soluble by this method; after all, $\BB$ is infinite-dimensional.

The difficulty in the implementation of $C^{-1}L_B$ is the appearance of the inverse of an upper-triangular banded matrix. Although inverses of banded matrices are efficiently represented by semi-separable matrices~\cite[Chap. 14]{Vandebril-Van-Barel-Mastronardi-1-08}, only the lower-triangular part of $C^{-1}L_B$ is required, as the upper-triangular part is populated by reflection. One computationally inexpensive scheme to obtain $\tril(C^{-1}L_B)$ creates a partial inverse of $C$ that is also upper-triangular and banded, whose upper bandwidth is necessarily and sufficiently the sum of the upper bandwidth of $C$, the lower bandwidth of $L_B$, and the lower bandwidth of $A$. Another is to solve the banded linear systems for the lower triangular entries of each and every column of $C^{-1}L_B$. This approach may be more stable for larger bandwidth problems.

\subsection{Skew-symmetry}

It is common for a second-order self-adjoint problem to be factored in terms of the product of skew-symmetric differentiation matrices on either side of a positive-definite multiplication operator representing the non-negative variable coefficient. Skew-symmetry of the differentiation matrices then results in semi-discretizations of equations of evolution that are ``stable by design''~\cite{Iserles-Webb-18}. For problems on $L^2([-1,1])$ such as~\cite{Driscoll-Hale-36-108-16}:
\begin{equation}
\DD u = \lambda u,\qquad u(-1) + u(1) = 0,
\end{equation}
the boundary condition is represented by the linear functional on normalized Legendre polynomials:
\begin{equation}
\BB = \begin{pmatrix} \sqrt{2} & 0 & \sqrt{10} & 0 & \sqrt{18} & \cdots\end{pmatrix},
\end{equation}
and our pathological method suggests we expand in the recombined polynomial basis:
\begin{equation}
\{\phi_n\} = \left\{\tilde{P}_0 - \sqrt{\tfrac{1}{5}}\tilde{P}_2, \tilde{P}_1, \tilde{P}_2 - \sqrt{\tfrac{5}{9}}\tilde{P}_4, \tilde{P}_3, \tilde{P}_4 - \sqrt{\tfrac{9}{13}}\tilde{P}_6,\ldots\right\}.
\end{equation}
It is left as an exercise for the reader to confirm the banded skew-symmetry of the resulting differentiation matrix and the banded symmetric positive-definiteness of the accompanying conversion between trial and test functions.

\section{Applications}\label{section:eigenvalueapplications}

The ultraspherical spectral method has been implemented in the software package {\tt ApproxFun.jl}~\cite{ApproxFun}. We have contributed to this package by adding the mechanism which symmetrizes the formally self-adjoint linear differential operators. Below, we show a variety of different applications that may help the reader in determining the applicability of the method. As the symmetrizing mechanism has been implemented in software, our discussion regarding the literal discretizations is rather kurt with the exception of one situation where it is described completely.

Numerically, it is impossible to determine whether or not a linear differential operator and its boundary conditions are formally self-adjoint. In the context of our symmetrizer and due to Eq.~\eqref{eq:PetrovGalerkinisBanded}, the entries of the symmetric-definite discretization are populated in the subdiagonal part of the operator. If the boundary conditions were not self-adjoint, then the entries in superdiagonals above the banded part of the operator would not be relatively near the machine precision. If the operator were not self-adjoint, then the entries above the main diagonal would not be conjugates of the associated entries below the main diagonal. Since linear differential operators may be unbounded, even determining a heuristic tolerance for either one of the conditions of bandedness or symmetry would be challenging. Thus, in contrast with the solution of linear differential equations, there is an additional burden on the user to ensure that her/his eigenproblem is formally self-adjoint.

\subsection{Asymptotics of the spectra of confined quantum anharmonic oscillators}

The Schr\"odinger equation~\cite{Schrodinger-28-1049-26} is a celebrated Sturm--Liouville eigenvalue problem:
\begin{equation}\label{eq:Schrodinger}
\left[-\DD^2 + V(x)\right]u = \lambda u,\qquad \lim_{x\to\pm\infty}u(x) = 0.
\end{equation}
In~\cite{Gaudreau-Slevinsky-Safouhi-337-261-13}, Gaudreau et al.~derive a reversed WKB series for the eigenvalues of the Schr\"odinger equation on the real line with the quartic anharmonic potential $V(x) = \omega x^2 + x^4$ of the form:
\begin{equation}
\sqrt{\lambda_n^{\rm quartic}} \sim \left(\frac{3\Gamma(\nicefrac{3}{4})(n+\nicefrac{1}{2})\sqrt{\pi}}{\Gamma(\nicefrac{1}{4})}\right)^{\nicefrac{2}{3}} + \sum_{k=0}^\infty\frac{F_k(\omega)}{\left[(n+\nicefrac{1}{2})\pi\right]^{\nicefrac{2k}{3}}},\quad{\rm as}\quad n\to\infty,
\end{equation}
where:
\begin{align}
F_0(\omega) & = \frac{\omega \Gamma(\nicefrac{3}{4})^4}{\pi^2},\\
F_1(\omega) & = -\frac{\sqrt[3]{6}\omega^2\left(\pi^4-4\Gamma(\nicefrac{3}{4})^8\right)}{48\pi^3\Gamma(\nicefrac{3}{4})^{\nicefrac{4}{3}}},\\
F_2(\omega) & = \frac{\sqrt[3]{36}\Gamma(\nicefrac{3}{4})^{\nicefrac{4}{3}}\left[(3\pi^4-4\Gamma(\nicefrac{3}{4})^8)\omega^3+12\pi^4\right]}{216\pi^4},
\end{align}
and more terms are available.

When the domain is truncated to $[-10,10]$ and Eq.~\eqref{eq:Schrodinger} is accompanied by Dirichlet boundary conditions, Weyl's law~\cite{Weyl-71-441-12} states that we should anticipate the eigenvalues with absolute value greater than the potential to approach those of a square well potential:
\begin{equation}
\lambda_n^{\rm square~well} \sim \left(\frac{\pi n}{\mu([-10,10])}\right)^2,\quad{\rm as}\quad n\to\infty,
\end{equation}
where $\mu([-10,10]) = 20$ is the Lebesgue measure of the interval.

Additionally, if $\omega\gg1$, we anticipate the smallest eigenvalues to have the asymptotic behaviour of the harmonic oscillator:
\begin{equation}
\lambda_n^{\rm quadratic} \sim \sqrt{\omega}(2n+1).
\end{equation}

With $\omega = 25$, all three of these asymptotic r\'egimes are on display in the left panel of Figure~\ref{fig:AnharmonicAndGeneralized} when viewing the first $10,000$ eigenvalues of Eq.~\eqref{eq:Schrodinger}. To calculate them, we create a $16,667\times16,667$ banded and symmetric-definite discretization, we compute the truncated eigenvalues, and we extract the lowest $60\%$, which we anticipate to be accurate based on the discounting factor of $\frac{\pi}{2}$ in the degree of Chebyshev/Legendre expansions required to resolve plane wavefunctions of the same frequency~\cite{Weideman-Trefethen-25-1279-88}.

\begin{figure}
\begin{center}
\begin{tabular}{cc}
\hspace*{-0.2cm}\includegraphics[width=0.53\textwidth]{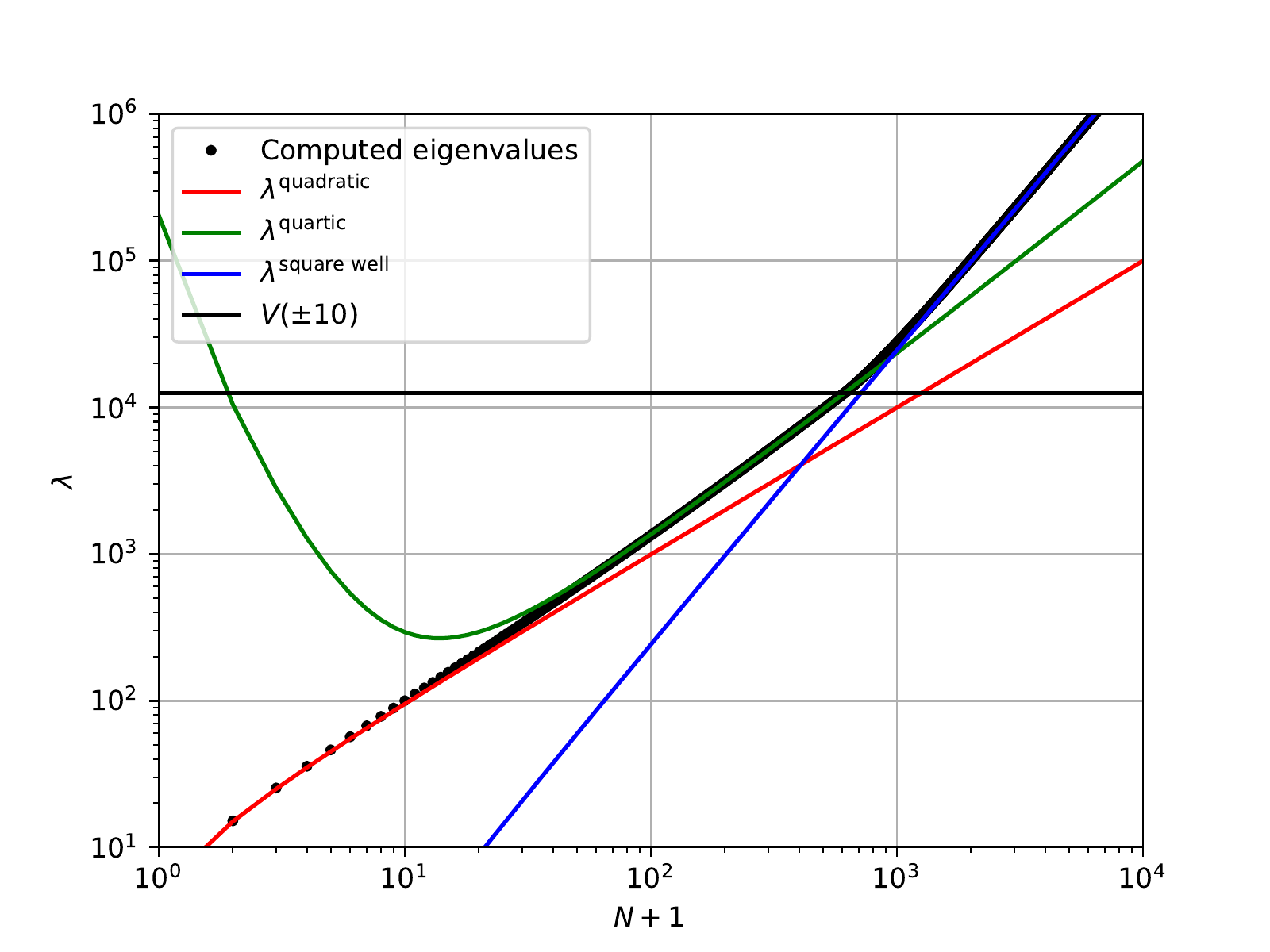}&
\hspace*{-0.65cm}\includegraphics[width=0.53\textwidth]{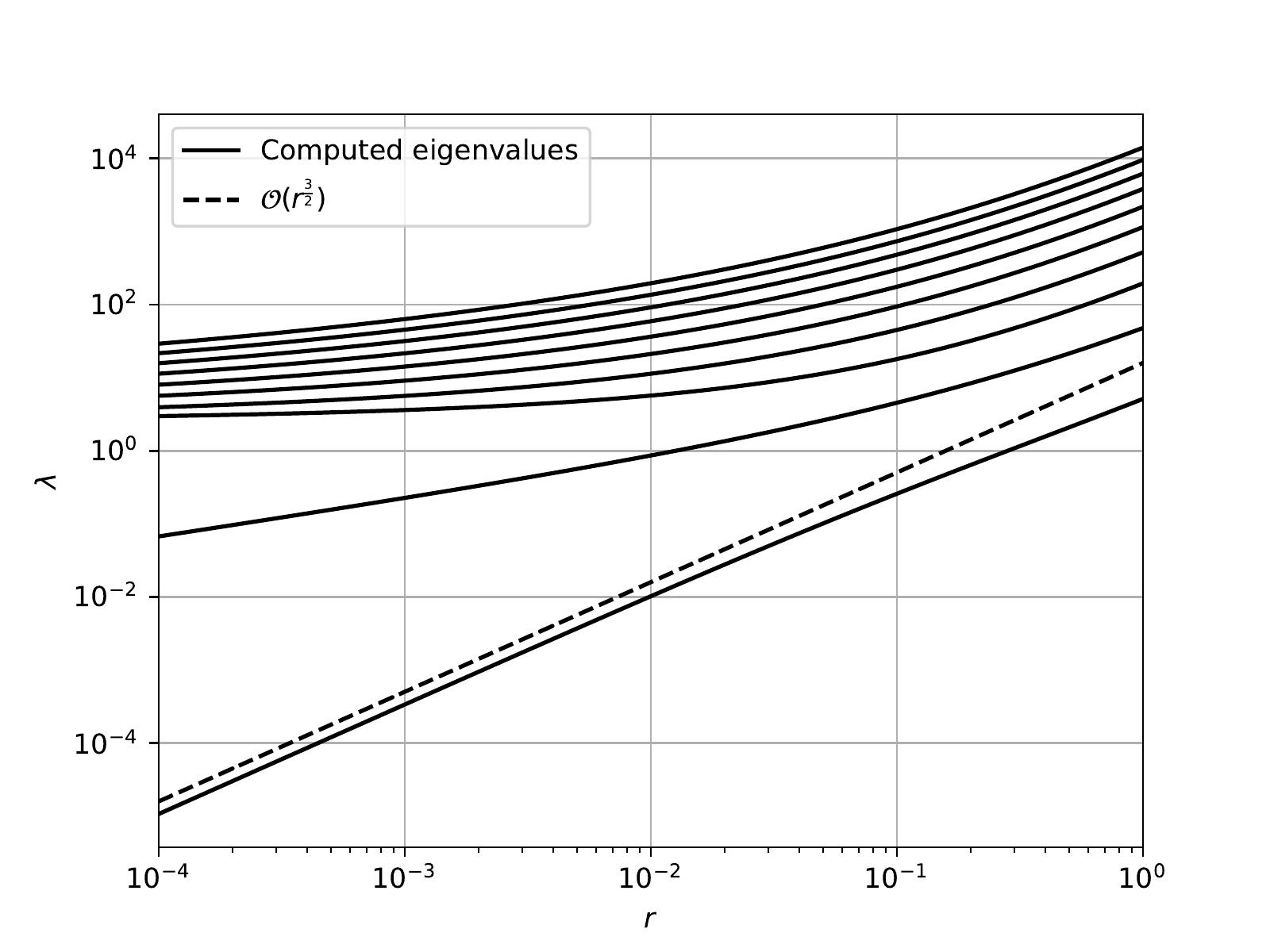}\\
\end{tabular}
\caption{Left: spectrum and asymptotic approximations of the anharmonic oscillator in Eq.~\eqref{eq:Schrodinger}. Right: the first $10$ generalized eigenvalues of Eq.~\eqref{eq:VibrationNonuniformRings} as functions of the thickness parameter $r$.}
\label{fig:AnharmonicAndGeneralized}
\end{center}
\end{figure}

\subsection{Generalized eigenvalue problems}

Our Petrov--Galerkin approach also symmetrizes banded discretizations of generalized eigenvalue problems. The following sixth order generalized eigenvalue problem models the vibration of nonuniform rings~\cite{Gutierrez-Laura-185-507-95}. In self-adjoint form\footnote{Formally self-adjoint problems need not be instantiated in self-adjoint form for the symmetry to transcend the discretization, but a more compact representation may be more efficient.}, the problem reads:
\begin{equation}\label{eq:VibrationNonuniformRings}
\left\{(-\DD)^3\left[\phi\DD^3\right] - (-\DD)^2\left[2\pi^2\phi\DD^2\right] + (-\DD)\left[\left(\pi^4\phi+\pi^2\phi''\right)\DD\right]\right\} u = \lambda\left\{(-\DD)\left[\pi^4f\DD\right] + \pi^6 f\right\}u,
\end{equation}
where $f(x) = -4(r-1)x^2+4(r-1)x+1$, for a thickness parameter $0<r<\infty$, $\phi = f^3$, and the boundary conditions are $u = u' = u''' = 0$ at both $x=0$ and $x=1$.

With sixth order derivatives in the first operator in Eq.~\eqref{eq:VibrationNonuniformRings}, this is by all accounts an ill-conditioned generalized eigenvalue problem. And as $r\searrow0$, this ill-conditioning is compounded by the fact that both linear operators become singular as the smallest generalized eigenvalue tends to zero. This is explained by the fact that the variable coefficients:
\begin{equation}
f(x) = 1+4(1-r)\left[(x-\nicefrac{1}{2})^2-\nicefrac{1}{4}\right],
\end{equation}
and $\phi$ have multiple roots at the midpoint of the computational domain when $r=0$, and some of these roots would split apart were $r$ permitted to cross the physical threshold and become negative.

For constant coefficients ($r=1$), Driscoll~\cite{Driscoll-229-5980-10} accurately computes the first four eigenvalues and eigenfunctions in $64$-bit floating-point arithmetic, reporting the square roots of the eigenvalues as the vibration frequencies. The results are included here in Table~\ref{table:SixthOrderDriscoll}.

\begin{table}[hbtp]
\caption{Driscoll's~\cite{Driscoll-229-5980-10} first four computed generalized eigenvalues of Eq.~\eqref{eq:VibrationNonuniformRings} with constant coefficients.}
\begin{center}
\begin{tabular}{c r @{.} l r @{.} l}
\sphline
$n$ & \multicolumn{2}{c}{$\sqrt{\lambda_n^{\rm Driscoll}}$} & \multicolumn{2}{c}{$\lambda_n^{\rm Driscoll}$} \\
\sphline
$0$ & 2&266~742~077 & 5&138~119~644\\
$1$ & 6&923~297~244 & 47&932~044~73\\
$2$ & 13&977~668~83 & 195&375~225~9\\
$3$ & 22&819~562~61 & 520&732~437~7\\
\sphline
\end{tabular}
\end{center}
\label{table:SixthOrderDriscoll}
\end{table}%

Using our symmetric-definite banded discretization, we compute the generalized eigenvalues in 256-bit extended precision. The banded sparsity is crucial in extended precision and for ill-conditioned problems such as this one because every flop is so much more costly. Using the general Rayleigh quotient iteration~\cite{Rayleigh-37,Parlett-28-679-74} for the pencil $(A,B)$:
\begin{equation}\label{eq:GRQI}
\lambda^{(k)} = \frac{\langle v^{(k)}, A v^{(k)}\rangle}{\langle v^{(k)}, B v^{(k)}\rangle},\quad\hbox{followed by}\quad v^{(k+1)} = \frac{(A-\lambda^{(k)}B)^{-1} v^{(k)}}{\norm{(A-\lambda^{(k)}B)^{-1} v^{(k)}}_2},
\end{equation}
our extended precision results in Table~\ref{table:SixthOrderEP} agree with every digit in Driscoll's (correctly rounded) computations.

\begin{table}[hbtp]
\caption{The first four computed eigenvalues of Eq.~\eqref{eq:VibrationNonuniformRings} in $256$-bit extended precision.}
\begin{center}
\begin{tabular}{c r @{.} l}
\sphline
$n$ & \multicolumn{2}{c}{$\lambda_n$} \\
\sphline
$0$ & 5&138~119~642~550~555~649~415~412~286~723~460~005~700~130~780~549~503~937~354~845~212~533~811~\ldots\\
$1$ & 47&932~044~726~525~070~683~741~246~754~121~353~113~209~603~021~890~092~332~934~543~893~845~719~\ldots\\
$2$ & 195&375~225~987~311~298~618~562~101~588~390~534~216~626~466~473~664~548~029~401~788~843~545~607~\ldots\\
$3$ & 520&732~437~610~822~993~012~167~361~990~447~367~561~584~023~509~960~784~102~391~822~269~243~419~\ldots\\
\sphline
\end{tabular}
\end{center}
\label{table:SixthOrderEP}
\end{table}%

The right panel of Figure~\ref{fig:AnharmonicAndGeneralized} plots the first $10$ generalized eigenvalues as functions of the thickness parameter $r$ computed by creating $1,600\times1,600$ discretizations in extended precision and using the general Rayleigh quotient to iteratively refine approximations based on the spectral data at a slightly larger thickness parameter; that is, a homotopy marches across the figure from right to left. A banded $LDL^\top$ factorization is used to solve the linear systems in Eq.~\eqref{eq:GRQI} in linear complexity in the discretization dimension. The discrete generalized eigenvalues are compared with those corresponding to $3,200\times3,200$ discretizations and the pointwise relative difference on the left side of the figure is less than $3\times 10^{-8}$, sufficient for plotting purposes. We remark that some of the first $10$ generalized eigenvalues of the $800\times800$ discretization agree with the $1,600\times1,600$ only to two digits, a testament to the severe ill-conditioning of this problem. Based on our numerical experiments, we conjecture that the smallest generalized eigenvalue scales with the power law:
\begin{equation}
\lambda_0(r) = \OO(r^{\frac{3}{2}})\quad{\rm as}\quad r\searrow0.
\end{equation}

\subsection{Piecewise-defined bases}

Eigenfunctions need not live in an unweighted Hilbert space. On the ray $D = [-1,\infty)$, we consider the Sturm--Liouville eigenvalue problem:
\begin{equation}\label{eq:WeightedSchrodinger}
\left[-\DD^2 + V(x)\right]u = \lambda(1+\abs{x})u,\qquad u(-1) = \lim_{x\to\infty}u(x) = 0.
\end{equation}
with the piecewise-defined function:
\begin{equation}
V(x) = \left\{\begin{array}{ccr} -1 & \for & -1\le x < 0,\\ x^2 & \for & 0\le x < \infty.\end{array}\right.
\end{equation}
Formally, $\frac{1}{1+\abs{x}}\left[-\DD^2 + V(x)\right]$ is a self-adjoint operator on $L^2([-1,\infty), (1+\abs{x})\ud x)$, and we could identify an orthonormal basis that spans this Hilbert space. However, division by $1+\abs{x}$ is only approximately representable as a banded linear operator on weighted polynomial bases, resulting in an undue complication in the discretization. Instead, we view both sides of Eq.~\eqref{eq:WeightedSchrodinger} as self-adjoint operators on $L^2([-1,\infty),\ud x)$ and we utilize the following piecewise-defined bases as the starting point for a banded symmetric-definite discretization:
\begin{equation}
\phi_n^L(x) = \sqrt{2}\tilde{P}_n(2x+1)\rchi_{[-1,0]}(x),\quad{\rm and}\quad \phi_n^R(x) = e^{-x/2}\tilde{L}_n(x)\rchi_{[0,\infty)}(x),
\end{equation}
where $\tilde{L}_n(x) \equiv L_n(x)$ are the orthonormalized Laguerre polynomials that happen to coincide with the standard normalization~\cite[\S 18.3]{Olver-et-al-NIST-10}, and $\rchi_I(x)$ is the characteristic function on the set $I$.

Notice that by incorporating half the full Laguerre weight, $e^{-x}$, the orthonormal basis $\{\phi_n^R\}$ already satisfies the boundary condition at infinity. However, by interlacing the left and right bases:
\begin{equation}
\{\phi_n\} = \{ \phi_0^L, \phi_0^R, \phi_1^L, \phi_1^R, \ldots \},
\end{equation}
we must also enforce two additional continuity conditions at the junction, namely $u(0^-) - u(0^+) = u'(0^-) - u'(0^+) = 0$, for a grand total of three infinite-dimensional linear algebraic conditions in $\BB$.

In this problem, an exponentially small error would be retained if the domain were truncated to a reasonable size, such as $[-1,50]$, and two piecewise-defined scaled and shifted Legendre polynomial bases were elected to span $L^2([-1,50], \ud x)$. However, recombinations of weighted Laguerre polynomial expansions allow one to recover high relative accuracy far along the tails of the eigenfunctions that may be utilized in sensitive subsequent calculations, such as the computation of conditional expectations.

Figure~\ref{fig:BandedSparsity} shows the banded sparsity patterns in the symmetric-definite discretizations of all three problems in this section.

\begin{figure}
\begin{center}
\hspace*{-0.2cm}\includegraphics[width=\textwidth]{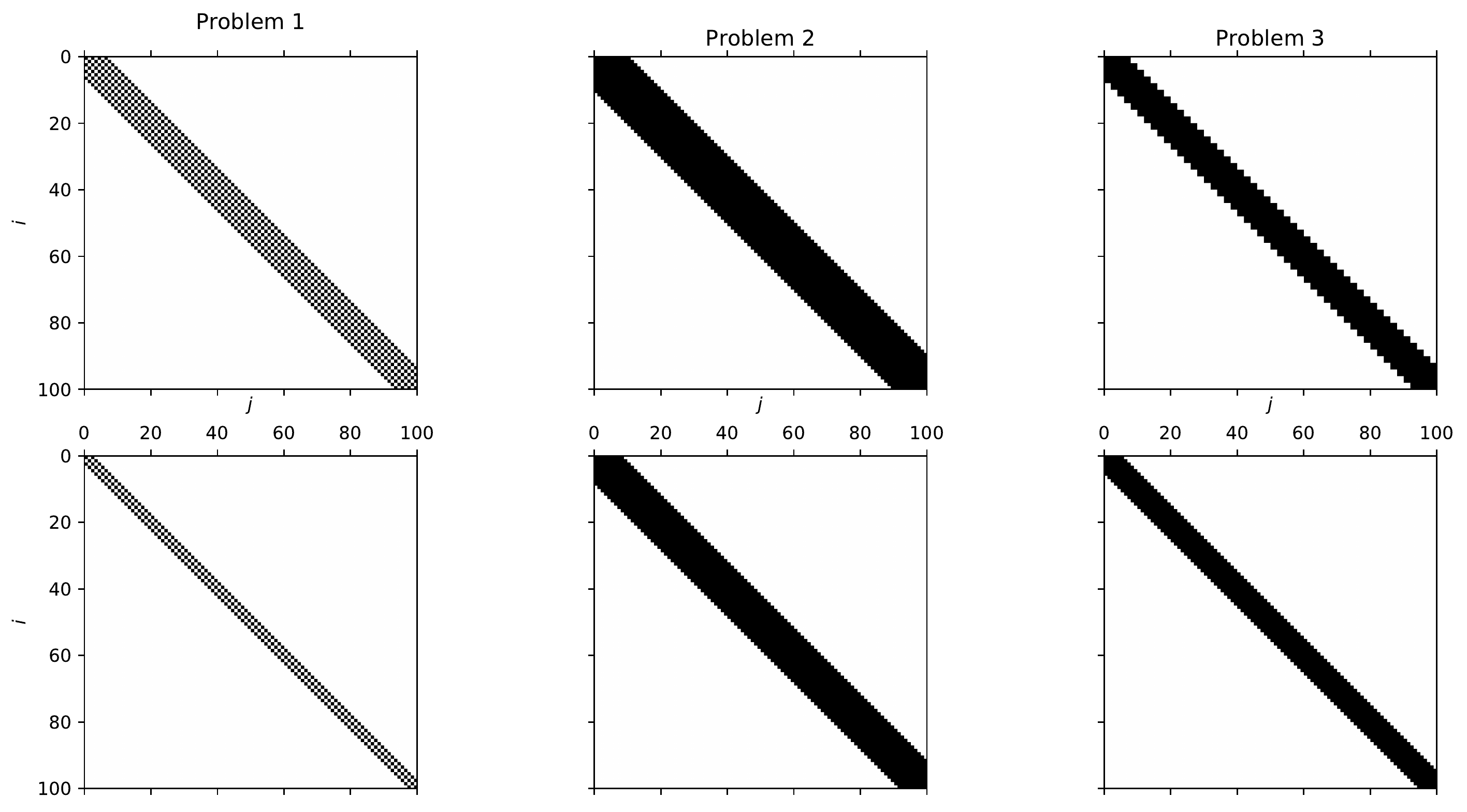}
\caption{Banded sparsity patterns of the symmetric-definite discretizations of the three problems in this section. The top row are the left-hand side operators and the bottom row represent the right-hand side operators. For the second problem, we illustrate the case $r=\nicefrac{1}{2}$ as the bandwidth when $r=1$ is exceptionally small.}
\label{fig:BandedSparsity}
\end{center}
\end{figure}

\subsection{Spectral analysis of the model problem}\label{subsection:spectralanalysis}

In this section, we consider a spectral analysis of the negative Laplacian on $[-1,1]$ with Dirichlet boundary conditions, as discussed in \S~\ref{section:modelproblem}. Our spectral discretization results in the generalized symmetric-definite pencil $(D,M)$, where:
\begin{equation}\label{eq:DMpencil}
D_{n,n} = (n+1)(n+2),\qquad M_{n,n} = \frac{2(n+1)(n+2)}{(2n+1)(2n+5)},\quad{\rm and}\quad M_{n,n+2} = M_{n+2,n} = -\sqrt{\frac{(n+1)(n+2)(n+3)(n+4)}{(2n+3)(2n+5)^2(2n+7)}},
\end{equation}
and all other entries are zero. By the type of pencil, all generalized eigenvalues are real. Since $M$ is irreducible, all generalized eigenvalues are distinct. The remaining challenge is to obtain individual bounds.

A first idea is to employ Gerschgorin's theorem, either on the pencil~\cite{Nakatsukasa-80-2127-11} or after congruence transformation to the standard eigenproblem $D^{-\frac{1}{2}}MD^{-\frac{1}{2}}$. Since $D$ is diagonal and positive, its square-root is real and agrees with the component-wise definition. However, most of the resulting Gerschgorin disks are not disjoint and so the approach fails to offer reasonable bounds. Furthermore, it is readily confirmed that diagonal similarity transformations cannot isolate each and every disk~\cite{Henrici-11-281-63,Porsching-8-437-66}.

We thus continue with a different approach based on Sylvester's law of inertia, which has been explored and extended to pencils by Nakatsukasa and Noferini~\cite{Nakatsukasa-Noferini-1711-00495}. Given a shift $\sigma\in\R$, one may determine how many generalized eigenvalues of $(D,M)$ are larger than, equal to, or smaller than $\sigma$ by examining the inertia of the matrix $D-\sigma M$; they are equal. From Eq.~\eqref{eq:DMpencil}, the left boundary of the $n^{\rm th}$ Gerschgorin disk of $D-\sigma M$ is:
\begin{equation}
\ell_n = (n+1)(n+2) - \sigma\left(\frac{2(n+1)(n+2)}{(2n+1)(2n+5)} + \sqrt{\frac{(n+1)(n+2)(n+3)(n+4)}{(2n+3)(2n+5)^2(2n+7)}} + \sqrt{\frac{(n-1)n(n+1)(n+2)}{(2n-1)(2n+1)^2(2n+3)}}\right).
\end{equation}
Since the term in parentheses is bounded above by $1$ provided that $n>0$, the left boundary satisfies:
\begin{equation}
\ell_n > (n+1)(n+2) - \sigma.
\end{equation}
It follows that if the mild shift $\sigma_n = (n+1)(n+2)$ is chosen, then the trailing infinite section $(D-\sigma M)_{n+1:\infty,n+1:\infty}$ is positive-definite, and hence $\lambda_n(D,M) > (n+1)(n+2)$. For an upper bound, by the Courant--Fisher minimax theorem~\cite[Theorem 8.1.2]{Golub-Van-Loan-13}, it suffices to consider the $(n+1)\times (n+1)$ principal finite sections of $D$ and $M$. Then:
\begin{equation}
\lambda_n(D,M) \le \norm{(M^{-1}D)_{0:n,0:n}}_2 \le \norm{(M^{-1})_{0:n,0:n}}_2\norm{D_{0:n,0:n}}_2.
\end{equation}
For estimates of the induced $2$-norms, we find the lower bound on the smallest Gerschgorin disk for $M_{0:n,0:n}$:
\begin{equation}
\frac{2(n+1)(n+2)}{(2n+1)(2n+5)} - \sqrt{\frac{(n+1)(n+2)(n+3)(n+4)}{(2n+3)(2n+5)^2(2n+7)}} - \sqrt{\frac{(n-1)n(n+1)(n+2)}{(2n-1)(2n+1)^2(2n+3)}} > \frac{1}{2(n+1)(n+2)},
\end{equation}
and since $D$ is diagonal, $\norm{D_{0:n,0:n}}_2 = (n+1)(n+2)$ so that:
\begin{equation}
\lambda_n(D,M) < 2\left[(n+1)(n+2)\right]^2.
\end{equation}
Finally, we note that the inequality $2 < \lambda_0(D,M) < 8$ also holds since the $0^{\rm th}$ Gerschgorin disk for $(D,M)$ is isolated.

The analysis of the eigenvalues of the negative Laplacian are important for understanding the stability of the choice of spatial discretizations when marching forward in time, e.g.~when solving the heat equation. It is also important as a model problem because all the eigenvalues are known analytically. For $n\times n$ pseudospectral discretizations in Chebyshev points, Gottlieb and Lustman~\cite{Gottlieb-Lustman-20-909-83} proved that the spectra are real, positive, and distinct, Charalambides and Waleffe~\cite{Charalambides-Waleffe-46-280-08} proved that the spectra of sequential discretizations interlace, and Weideman and Trefethen~\cite{Weideman-Trefethen-25-1279-88} proved an asymptotically constant lower-bound and quartic upper-bound for all of them. These results are impressive for a dense asymmetric matrix. But in the limit of a large number of collocation points, do they guarantee that the pseudospectral discretization has discrete spectrum that accumulates at infinity? Has continuous spectrum been ruled out? The preceding analysis of our symmetric-definite and banded discretization reveals independent bounds for each eigenvalue such that our discretization is faithful to the continuous operator all the way to infinity.

More generally, we argue that banded sparsity promotes a deeper, more penetrating, analysis that is equally applicable to the other problems in this section; please see Figure~\ref{fig:BandedSparsity}. While the analysis may be tedious, the sparsity in our discretization renders it in principle much simpler than many alternative (dense) discretizations.

\begin{remark}
The diagonal entries, $d_n$, of $D$ in an $LDL^\top$ factorization of $D-\sigma M$ are given recursively by:
\begin{equation}
d_0 = D_{0,0} - \sigma M_{0,0},\quad d_1 = D_{1,1} - \sigma M_{1,1},\quad d_{n+2} = D_{n+2,n+2}-\sigma M_{n+2,n+2} - \frac{(\sigma M_{n+2,n})^2}{d_n}.
\end{equation}
Properties of the inertia of $D-\sigma M$ are obscured by the nonlinearity in this recurrence relation, but we should not preclude tighter bounds from being determined by this avenue.
\end{remark}

\section{Symmetric-(definite) banded eigensolvers}

So far, our applications have only considered eigenvalue computations. From here to the end, we consider algorithms for the spectral decomposition of self-adjoint problems.

In this section and the next, we break with the notation already in use. In particular, regular capital letters represent finite-dimensional matrices, calligraphic letters represent infinite-dimensional matrices (rather than continuous operators), and indexing begins at $1,1$ rather than $0,0$. We also use shorthand notation of $a:b$ to represent all entries of an array from index $a$ to $b$ inclusive, and this is sometimes tensorized. Other notational shorthands attempt to keep the presentation light.

\begin{proposition}\label{proposition:SB_spectral}
Let $A\in\R^{k\times k}$ be a symmetric banded matrix with bandwidth $b>0$ such that $A_{i,j} = 0$ for $\abs{i-j} > b$. The algorithms of~\cite{Schwarz-12-231-68,Kaufman-26-551-00} compute the spectral decomposition:
\begin{equation}
A = Q\Lambda Q^\top,
\end{equation}
in $\OO(k^2b)$ flops. The eigenvectors $Q$ are products of Givens rotations and the eigenvectors of the similar symmetric tridiagonal problem, and the dense matrices may be formed in $\OO(k^3)$ flops if required.
\end{proposition}
The algorithms reduce the symmetric banded problem to symmetric tridiagonal form and introduce orthogonal similarity transformations to preserve the bandwidth of the problem, chasing bulges beyond the bandwidth off the matrix. These algorithms are backward stable, and Kaufman improves the efficiency by demonstrating that multiple orthogonal similarity transformations may be vectorized. The symmetric tridiagonal spectral decomposition may be produced using one of many alternatives in $\OO(k^2)$ flops~\cite{Golub-Van-Loan-13}.

\begin{proposition}\label{proposition:SB_generalized_spectral}
Let $A\in\R^{k\times k}$ and $B\in\R^{k\times k}$ be symmetric banded matrices both bandlimited by $b>0$, and let $B$ be positive-definite. The algorithms of~\cite{Crawford-16-41-73,Wilkinson-Jacobs-77,Kaufman-14-372-93} compute the generalized spectral decomposition:
\begin{equation}
V^\top A V = \Lambda,
\end{equation}
in $\OO(k^2b)$ flops. In this case, the eigenvectors are $B$-orthogonal, satisfying $V^\top B V = I$.
\end{proposition}
In LAPACK, this algorithm uses the split-Cholesky factorization of $B = S^\top S$ to convert the symmetric-definite banded generalized eigenvalue problem into the regular symmetric banded eigenvalue problem:
\begin{equation}
X^\top A X = C,\qquad X^\top B X = I.
\end{equation}
The matrix $X := S^{-1}Q$ consists of the inverse of one of the split-Cholesky factors combined with orthogonal transformations that are designed to ensure that $C$ is banded. This improvement due to Wilkinson~\cite{Wilkinson-Jacobs-77} is based on Crawford's algorithm~\cite{Crawford-16-41-73} which originally uses a Cholesky factorization of $B = LL^\top$.

Afterward, the symmetric banded spectral decomposition of $C$ follows from Proposition~\ref{proposition:SB_spectral}:
\begin{equation}
C = W\Lambda W^\top.
\end{equation}
In terms of $S$, $Q$, and $W$, the generalized eigenvectors are given by the product $V = S^{-1}QW$.

The stability of the algorithms is related to the norms of the split-Cholesky factors. Thus, ill-conditioning in $B$ may obscure the result.

\section{An adaptive spectral decomposition}

Several authors have recently investigated algorithms for infinite-dimensional numerical linear algebra, including an adaptive $QR$ decomposition by Olver and Townsend~\cite{Olver-Townsend-55-462-13}, algorithms for the spectra of finite-rank perturbations of Toeplitz symmetric tridiagonal Jacobi operators by Webb~\cite{Webb-Thesis-17}, and related algorithms investigated by Hansen~\cite{Hansen-254-2092-08} and Colbrook and Hansen~\cite{Colbrook-Hansen-IQR-18}. In infinite dimensions, similarity and congruence transformations may begin ``at infinity'' and terminate at the base of the operator, or they may begin at the base, iteratively working on finite-sections of the operator. In our adaptive spectral decomposition below, we espouse the latter perspective.

\subsection{An adaptive spectral decomposition}

In this section, we describe an algorithm for the partial spectral decomposition of:
\begin{equation}
\LL u = \lambda u,
\end{equation}
with bandwidth $b$.

\begin{algorithm}\label{algorithm:SB_AED}
Initialize $k := 2b$. Partition $\LL$ into:
\begin{equation}
\LL = \begin{pmatrix}
\LL_{1:k,1:k} & \LL_{k+1:\infty,1:k}^\top\\
\LL_{k+1:\infty,1:k} & \LL_{k+1:\infty,k+1:\infty}\\
\end{pmatrix}
=: \begin{pmatrix}
L_k & L_b^\top\\
L_b & \LL_{k+1:\infty}\\
\end{pmatrix}
\end{equation}
In $\OO(k^2b)$ operations, compute the symmetric banded spectral decomposition:
\begin{equation}
L_k = Q\Lambda Q^\top,
\end{equation}
and use the orthogonal matrix $Q$ to rewrite the partition as a similarity transformation on the {\em operator}:
\begin{equation}\label{eq:similarL}
\LL = \begin{pmatrix} Q &\\ & \II \end{pmatrix}
\begin{pmatrix} \Lambda & Q^\top L_b^\top\\
L_bQ & \LL_{k+1:\infty}\\
\end{pmatrix}
\begin{pmatrix} Q^\top &\\ & \II \end{pmatrix}.
\end{equation}
Define $M := L_bQ$ and set $j$ to be the largest integer multiple of $b$ that satisfies the inequality:
\begin{equation}\label{eq:SB_AED_deflationcriterion}
\norm{M_{1:b,1:j}}_F^2 \le {\tt tol} \left( \abs{\lambda_{\rm max}(\Lambda_{1:j,1:j})} + \abs{\lambda_{\rm min}(\Lambda_{1:j,1:j})} \right) \left( \abs{\lambda_{\rm max}(\Lambda_{1:j,1:j})} + \abs{\lambda_{\rm max}(\Lambda)} \right),
\end{equation}
for some tolerance ${\tt tol} > 0$. For eigenvalues only, ${\tt tol} = \epsilon$ and for both eigenvalues and eigenfunctions, ${\tt tol = \epsilon^2}$. If there is no integer $j$ that satisfies the inequality, then double $k$ and reiterate until deflation sets in.

After determining that $j$ eigenvalues are to be deflated, we still have the operator partitioned as above. Since $M$ is dense, the structure of the similar operator is now:
\begin{equation}\label{eq:firstpartition}
\begin{pmatrix} \Lambda & M^\top\\
M & \LL_{k+1:\infty}\\
\end{pmatrix}
=
\arraycolsep=1.4pt
\begin{pmatrix} \mydiagdown & & & & & & \square & &\\
& \mydiagdown & & & & & \square & &\\
& & \mydiagdown & & & & \square & &\\
& & & \mydiagdown & & & \square & &\\
& & & & \myddots & & \myvdots & &\\
& & & & & \mydiagdown & \square & &\\
\square & \square & \square & \square & \mycdots & \square & \square & \lltriangle &\\
& & & & & & \urtriangle & \square & \myddots\\
& & & & & & & \myddots & \myddots\\
\end{pmatrix}
\end{equation}
We perturb the operator by setting $M_{1:b,1:j} \leftarrow 0$. Thus:
\begin{equation}
\begin{pmatrix} \Lambda & M^\top\\
M & \LL_{k+1:\infty}\\
\end{pmatrix}
\approx
\arraycolsep=1.4pt
\begin{pmatrix} \mydiagdown & & & & & & 0 & &\\
& \myddots & & & & & \myvdots & &\\
& & \mydiagdown & & & & 0 & &\\
& & & \mydiagdown & & & \square & &\\
& & & & \myddots & & \myvdots & &\\
& & & & & \mydiagdown & \square & &\\
0 & \mycdots & 0 & \square & \mycdots & \square & \square & \lltriangle &\\
& & & & & & \urtriangle & \square & \myddots\\
& & & & & & & \myddots & \myddots\\
\end{pmatrix}
\end{equation}
It is now clear that the first $j$ eigenvalues of the perturbed operator may be deflated. Furthermore, the remaining part of the similar operator is of the same form:
\begin{equation}
\arraycolsep=1.4pt
\begin{pmatrix} \mydiagdown & & & & & & \square & &\\
& \mydiagdown & & & & & \square & &\\
& & \mydiagdown & & & & \square & &\\
& & & \mydiagdown & & & \square & &\\
& & & & \myddots & & \myvdots & &\\
& & & & & \mydiagdown & \square & &\\
\square & \square & \square & \square & \mycdots & \square & \square & \lltriangle &\\
& & & & & & \urtriangle & \square & \myddots\\
& & & & & & & \myddots & \myddots\\
\end{pmatrix}
\end{equation}
We now perform orthogonal similarity transformations to return the operator to symmetric banded form. Below, we describe the effect of applying Householder reflectors~\cite[\S 5.1]{Golub-Van-Loan-13} from the left. The analogous Householder reflectors are also applied from the right, and the figures show the effect of both as similarity transformations. To simplify the exposition, let $\ell := k-j$.

\begin{enumerate}
\item Begin by applying Householder reflectors to lower triangularize the entries in rows $1:b$ and columns $\ell+1:\ell+b$, resulting in the similar operator:
\begin{equation}
\arraycolsep=1.4pt
\begin{pmatrix} \blacksquare & & & & & & \llblacktriangle & &\\
& \mydiagdown & & & & & \square & &\\
& & \mydiagdown & & & & \square & &\\
& & & \mydiagdown & & & \square & &\\
& & & & \myddots & & \myvdots & &\\
& & & & & \mydiagdown & \square & &\\
\urblacktriangle & \square & \square & \square & \mycdots & \square & \square & \lltriangle &\\
& & & & & & \urtriangle & \square & \myddots\\
& & & & & & & \myddots & \myddots\\
\end{pmatrix}
\end{equation}
Black symbols represent the cumulative modification of nonzero submatrices.
\item Next, apply Householder reflectors to lower triangularize the entries in rows $1:2b$ and columns $\ell+1:\ell+b$:
\begin{equation}
\arraycolsep=1.4pt
\begin{pmatrix} \blacksquare & \blacksquare & & & & & & &\\
\blacksquare & \blacksquare & & & & & \llblacktriangle & &\\
& & \mydiagdown & & & & \square & &\\
& & & \mydiagdown & & & \square & &\\
& & & & \myddots & & \myvdots & &\\
& & & & & \mydiagdown & \square & &\\
& \urblacktriangle & \square & \square & \mycdots & \square & \square & \lltriangle &\\
& & & & & & \urtriangle & \square & \myddots\\
& & & & & & & \myddots & \myddots\\
\end{pmatrix}
\end{equation}
\item Before continuing to annihilate entries in the $\ell+1:\ell+b$ columns, apply Householder reflectors to lower triangularize the entries in rows $1:b$ and columns $b+1:2b$:
\begin{equation}
\arraycolsep=1.4pt
\begin{pmatrix} \blacksquare & \llblacktriangle & & & & & & &\\
\urblacktriangle & \blacksquare & & & & & \llblacktriangle & &\\
& & \mydiagdown & & & & \square & &\\
& & & \mydiagdown & & & \square & &\\
& & & & \myddots & & \myvdots & &\\
& & & & & \mydiagdown & \square & &\\
& \urblacktriangle & \square & \square & \mycdots & \square & \square & \lltriangle &\\
& & & & & & \urtriangle & \square & \myddots\\
& & & & & & & \myddots & \myddots\\
\end{pmatrix}
\end{equation}
\item Next, apply Householder reflectors to lower triangularize the entries in rows $b+1:3b$ and columns $\ell+1:\ell+b$:
\begin{equation}
\arraycolsep=1.4pt
\begin{pmatrix} \blacksquare & \blacksquare & \llblacktriangle & & & & & &\\
\blacksquare & \blacksquare & \blacksquare & & & & & &\\
\urblacktriangle & \blacksquare & \blacksquare & & & & \llblacktriangle & &\\
& & & \mydiagdown & & & \square & &\\
& & & & \myddots & & \myvdots & &\\
& & & & & \mydiagdown & \square & &\\
& & \urblacktriangle & \square & \mycdots & \square & \square & \lltriangle &\\
& & & & & & \urtriangle & \square & \myddots\\
& & & & & & & \myddots & \myddots\\
\end{pmatrix}
\end{equation}
Note that by lower triangularizing in the previous step, we extend entries into the rows $1:b$ and columns $2b+1:3b$ as a {\em lower triangular submatrix}.
\item Apply Householder reflectors to lower triangularize the entries in rows $1:2b$ and columns $2b+1:3b$:
\begin{equation}
\arraycolsep=1.4pt
\begin{pmatrix} \blacksquare & \blacksquare & & & & & & &\\
\blacksquare & \blacksquare & \llblacktriangle & & & & & &\\
& \urblacktriangle & \blacksquare & & & & \llblacktriangle & &\\
& & & \mydiagdown & & & \square & &\\
& & & & \myddots & & \myvdots & &\\
& & & & & \mydiagdown & \square & &\\
& & \urblacktriangle & \square & \mycdots & \square & \square & \lltriangle &\\
& & & & & & \urtriangle & \square & \myddots\\
& & & & & & & \myddots & \myddots\\
\end{pmatrix}
\end{equation}
And apply Householder reflectors to lower triangularize the entries in rows $1:b$ and columns $b+1:2b$:
\begin{equation}
\arraycolsep=1.4pt
\begin{pmatrix} \blacksquare & \llblacktriangle & & & & & & &\\
\urblacktriangle & \blacksquare & \llblacktriangle & & & & & &\\
& \urblacktriangle & \blacksquare & & & & \llblacktriangle & &\\
& & & \mydiagdown & & & \square & &\\
& & & & \myddots & & \myvdots & &\\
& & & & & \mydiagdown & \square & &\\
& & \urblacktriangle & \square & \mycdots & \square & \square & \lltriangle &\\
& & & & & & \urtriangle & \square & \myddots\\
& & & & & & & \myddots & \myddots\\
\end{pmatrix}
\end{equation}
\item Continuing, we apply Householder reflectors to lower triangularize the entries in rows $2b+1:4b$ and columns $\ell+1:\ell+b$:
\begin{equation}
\arraycolsep=1.4pt
\begin{pmatrix} \blacksquare & \llblacktriangle & & & & & & &\\
\urblacktriangle & \blacksquare & \blacksquare & \llblacktriangle & & & & &\\
& \blacksquare & \blacksquare & \blacksquare & & & & &\\
& \urblacktriangle & \blacksquare & \blacksquare & & & \llblacktriangle & &\\
& & & & \myddots & & \myvdots & &\\
& & & & & \mydiagdown & \square & &\\
& & & \urblacktriangle & \mycdots & \square & \square & \lltriangle &\\
& & & & & & \urtriangle & \square & \myddots\\
& & & & & & & \myddots & \myddots\\
\end{pmatrix}
\end{equation}
\item At this point, we are ready to paraphrase the remaining orthogonal similarity transformations. The bulge that has been created in the banded submatrix is chased back up by Householder reflectors toward the top left corner to return to banded form:
\begin{equation}
\arraycolsep=1.4pt
\begin{pmatrix} \blacksquare & \llblacktriangle & & & & & & &\\
\urblacktriangle & \blacksquare & \llblacktriangle & & & & & &\\
& \urblacktriangle & \blacksquare & \llblacktriangle & & & & &\\
& & \urblacktriangle & \blacksquare & & & \llblacktriangle & &\\
& & & & \myddots & & \myvdots & &\\
& & & & & \mydiagdown & \square & &\\
& & & \urblacktriangle & \mycdots & \square & \square & \lltriangle &\\
& & & & & & \urtriangle & \square & \myddots\\
& & & & & & & \myddots & \myddots\\
\end{pmatrix}
\end{equation}
\item Finally, further sequences of Householder reflectors continue to annihilate entries in the $\ell+1:\ell+b$ columns, chasing bulges in the banded submatrix in between, until a we arrive at a similar banded operator:
\begin{equation}
\arraycolsep=1.4pt
\begin{pmatrix} \blacksquare & \llblacktriangle & & & & & & &\\
\urblacktriangle & \blacksquare & \llblacktriangle & & & & & &\\
& \urblacktriangle & \blacksquare & \llblacktriangle & & & & &\\
& & \urblacktriangle & \blacksquare & \llblacktriangle & & & &\\
& & & \urblacktriangle & \blacksquare & \myddots & & &\\
& & & & \myddots & \myddots & \llblacktriangle & &\\
& & & & & \urblacktriangle & \square & \lltriangle &\\
& & & & & & \urtriangle & \square & \myddots\\
& & & & & & & \myddots & \myddots\\
\end{pmatrix}
\end{equation}
\end{enumerate}
\end{algorithm}

\subsection{Complexity analysis}

Based on the results in~\cite{Schwarz-12-231-68,Kaufman-26-551-00}, this algorithm requires:
\begin{enumerate}
\item $\OO(k^2b)$ flops for the symmetric banded spectral decomposition $L_k = Q\Lambda Q^\top$;

\item $\OO(kb^2)$ flops for the banded matrix-matrix product $M = L_b Q$ ($L_b$ has at most $b(b+1)/2$ nontrivial entries);

\item $\OO(j b)$ flops for computing the deflation criterion; and,

\item $\OO(\ell b^2)$ flops for folding in the symmetric spikes and $\OO(\ell^2 b)$ for chasing the symmetric bulges up and off the band.
\end{enumerate}

To ensure quadratic complexity, we must ensure the number of deflated eigenvalues in each iteration is not too small relative to the dimensions of the deflation window. To make this precise, let $j_1$ be the number of deflated eigenvalues after one step and let $j_2$ be the number after another. The total cost is:
\begin{equation}
\hbox{Two Step Cost} = 2 C_1 k^2 b + 2 C_2 kb^2 + C_3 (j_1+j_2) b + C_4 ( (k-j_1)^2 + (k-j_2)^2 ) b + C_5 ( (k-j_1) + (k-j_2) ) b^2,
\end{equation}
where $C_i$, for $i=1:5$, are positive constants that depend on the implementation.

Compare this with the same algorithm on the larger window of dimensions $(k+j_1) \times (k+j_1)$. This window theoretically deflates the same number of eigenvalues as two consecutive steps of fixed dimensions $k \times k$:
\begin{equation}
\hbox{One Large Step Cost} = C_1 (k+j_1)^2 b + C_2 (k+j_1) b^2 + C_3 (j_1+j_2) b + C_4 (k-j_2)^2 b + C_5 (k-j_2) b^2.
\end{equation}
If we assume that $C_1 \gg C_2+C_3+C_4+C_5$, then this implies that the spectral decomposition is more expensive than computing the banded matrix-matrix multiplication, deflation criterion, bulge chasing, and spike folding. With this assumption, for the two step cost to break even with the larger one step cost, the number of eigenvalues that must be deflated is:
\begin{equation}
j_1 \ge (\sqrt{2} - 1) k \approx 0.41 k.
\end{equation}
If this is not the case, then we double the dimensions of the deflation window.

\subsection{Determining the deflation criterion}

In aggressive early deflation, it is important to identify which eigenvalues have converged in the working precision, subject to rounding errors. Based on Eq.~\eqref{eq:similarL}, partition the similar operator as:
\begin{equation}\label{eq:3x3partition}
\begin{pmatrix} \Lambda & M^\top\\
M & \LL_{k+1:\infty}\\
\end{pmatrix}
= \begin{pmatrix}
\Lambda_j & 0 & M_j^\top\\
0 & \Lambda_k & M_k^\top\\
M_j & M_k & \LL_\infty
\end{pmatrix},
\end{equation}
where $M_j = M_{1:\infty,1:j}$, $M_k = M_{1:\infty,j+1:k}$, $\Lambda_j = \Lambda_{1:j,1:j}$, $\Lambda_k = \Lambda_{j+1:k,j+1:k}$, and $\LL_\infty = \LL_{k+1:\infty}$ for brevity.
To determine the deflation criterion, we appeal to the block determinant formul\ae:
\begin{equation}
\det\begin{pmatrix} A & B\\ C & D\end{pmatrix} = \det(A)\det(D - CA^{-1}B) = \det(D)\det(A-BD^{-1}C),
\end{equation}
provided the inverses exist. If we use the formul\ae~{\em twice} reflecting the partition in Eq.~\eqref{eq:3x3partition}, then:
\begin{align}
\det(\LL-\lambda \II) & = \det(\LL_\infty - \lambda\II) \det(\Lambda - \lambda I - M^\top (\LL_\infty-\lambda\II)^{-1} M),\nonumber\\
& = \det(\LL_\infty - \lambda\II) \det\left[\begin{pmatrix} \Lambda_j - \lambda I &\\ & \Lambda_k - \lambda I\end{pmatrix} - \begin{pmatrix} M_j^\top (\LL_\infty-\lambda\II)^{-1} M_j & M_j^\top (\LL_\infty-\lambda\II)^{-1} M_k\\ M_k^\top (\LL_\infty-\lambda\II)^{-1} M_j & M_k^\top (\LL_\infty-\lambda\II)^{-1} M_k \end{pmatrix} \right],\nonumber\\
& = \det(\LL_\infty - \lambda\II)\det\left(\Lambda_j-\lambda I -M_j^\top(\LL_\infty - \lambda I)^{-1} M_j\right)\nonumber\\
& \times \det\Big[\Lambda_k-\lambda I - M_k^\top(\LL_\infty - \lambda I)^{-1}M_k\nonumber\\
& \quad - M_k^\top(\LL_\infty - \lambda I)^{-1}M_j\left( \Lambda_j-\lambda I - M_j^\top(\LL_\infty - \lambda I)^{-1}M_j\right)^{-1} M_j^\top(\LL_\infty - \lambda I)^{-1}M_k\Big].\label{eq:SB_AED_blockdeterminant}
\end{align}
The key to understanding Eq.~\eqref{eq:SB_AED_blockdeterminant} is the second determinant in the product:
\begin{equation}
\det\left(\Lambda_j-\lambda I -M_j^\top(\LL_\infty - \lambda I)^{-1} M_j\right).
\end{equation}
This determinant informs us how close the first $j$ eigenvalues are to deflation, for if $M_j \equiv 0$, then complete deflation is admissible. Numerically, we deflate when $M_j^\top (\LL_\infty - \lambda I)^{-1} M_j$ is relatively small compared with $\Lambda_j-\lambda I$. Since the operator is self-adjoint, we are interested in the relative norms over the spectrum of $\Lambda_j$:
\begin{equation}
\max_{\lambda_{\rm min}(\Lambda_j) \le \lambda \le \lambda_{\rm max}(\Lambda_j)} \norm{\Lambda_j-\lambda I} = \abs{\lambda_{\rm max}(\Lambda_j) - \lambda_{\rm min}(\Lambda_j)} \le \abs{\lambda_{\rm max}(\Lambda_j)} + \abs{\lambda_{\rm min}(\Lambda_j)},
\end{equation}
for any induced $p$-norm. On the other hand, the analysis of:
\begin{align*}
\max_{\lambda_{\rm min}(\Lambda_j) \le \lambda \le \lambda_{\rm max}(\Lambda_j)} \norm{M_j^\top(\LL_\infty - \lambda I)^{-1} M_j}  & \le \norm{M_j^\top}\norm{M_j} \max_{\lambda_{\rm min}(\Lambda_j) \le \lambda \le \lambda_{\rm max}(\Lambda_j)} \norm{(\LL_\infty - \lambda I)^{-1}},\\
& = \norm{M_j}^2 \max_{\lambda_{\rm min}(\Lambda_j) \le \lambda \le \lambda_{\rm max}(\Lambda_j)} \norm{(\LL_\infty - \lambda I)^{-1}}.
\end{align*}
is simplified in the induced $2$-norm, for the norm of the inverse of the infinite-dimensional operator $\LL_\infty - \lambda I$ may be rephrased in terms of the smallest eigenvalue of $(\LL_\infty - \lambda I)^{-1}$ for nearest $\lambda$, that is, the separation distance between the spectra of $\Lambda_j$ and $\LL_\infty$.

Using the above analysis, we derive the heuristic deflation criterion in Eq.~\eqref{eq:SB_AED_deflationcriterion} by assuming that the smallest eigenvalue of $\LL_\infty$ is larger than the largest eigenvalue of $\Lambda$, in absolute value, and by using the equivalence of (finite-dimensional) Frobenius and induced $2$-norms. The Frobenius norm simplifies the computation, since no singular values need be computed, and since the square of the $2$-norm of the $j^{\rm th}$ column of $M$ may be added to $\norm{M_{1:b,1:j-1}}_F^2$ recursively:
\begin{equation}
\norm{M_{1:b,1:j}}_F^2 = \norm{M_{1:b,1:j-1}}_F^2 + \norm{M_{1:b,j}}_2^2.
\end{equation}

\subsection{An adaptive generalized spectral decomposition}

As in the applications of \S~\ref{section:eigenvalueapplications}, our Petrov--Galerkin scheme for a regular eigenvalue problem may require reformulation as a generalized eigenvalue problem:
\begin{equation}
\AA u = \lambda \BB u.
\end{equation}
We will assume that the $b$ is the maximum of the bandwidths of $\AA$ and $\BB$ and that $\BB$ is also positive-definite.

\begin{algorithm}\label{algorithm:SDB_AED}
Initialize $k := 2b$. Partition $\AA$ into:
\begin{equation}
\AA = \begin{pmatrix}
\AA_{1:k,1:k} & \AA_{k+1:k+b,1:k}^\top & 0\\
\AA_{k+1:k+b,1:k} & \AA_{k+1:k+b,k+1:k+b} & \AA_{k+b+1:\infty,k+1:k+b}^\top\\
0 & \AA_{k+b+1:\infty,k+1:k+b} & \AA_{k+b+1:\infty,k+b+1:\infty}
\end{pmatrix}
=: \begin{pmatrix}
A_k & A_b^\top & 0\\
A_b & A_{kb} & A_{2b}^\top\\
0 & A_{2b} & \AA_{k+b+1:\infty}\\
\end{pmatrix},
\end{equation}
and partition $\BB$ conformably:
\begin{equation}
\BB = \begin{pmatrix}
\BB_{1:k,1:k} & \BB_{k+1:k+b,1:k}^\top & 0\\
\BB_{k+1:k+b,1:k} & \BB_{k+1:k+b,k+1:k+b} & \BB_{k+b+1:\infty,k+1:k+b}^\top\\
0 & \BB_{k+b+1:\infty,k+1:k+b} & \BB_{k+b+1:\infty,k+b+1:\infty}
\end{pmatrix}
=: \begin{pmatrix}
B_k & B_b^\top & 0\\
B_b & B_{kb} & B_{2b}^\top\\
0 & B_{2b} & \BB_{k+b+1:\infty}\\
\end{pmatrix}.
\end{equation}
In $\OO(k^2b)$ operations, compute the symmetric-definite banded generalized spectral decomposition:
\begin{equation}
V^\top A_k V = \Lambda,\quad{\rm where}\quad V^\top B_k V = I,
\end{equation}
and use the eigenvectors $V$ to rewrite the partition as a congruence transformation on the {\em operators}:
\begin{align*}
\AA & = \begin{pmatrix} V^{-\top} & & \\ & I & \\ & & \II \end{pmatrix}
\begin{pmatrix} \Lambda & V^\top A_b^\top & 0\\
A_bV & A_{kb} & A_{2b}^\top \\
0 & A_{2b} & \AA_{k+b+1:\infty}
\end{pmatrix}
\begin{pmatrix} V^{-1} & & \\ & I & \\ & & \II \end{pmatrix},\quad{\rm and}\\~\\
\BB & = \begin{pmatrix} V^{-\top} & & \\ & I & \\ & & \II \end{pmatrix}
\begin{pmatrix} I & V^\top B_b^\top & 0\\
B_bV & B_{kb} & B_{2b}^\top \\
0 & B_{2b} & \BB_{k+b+1:\infty}
\end{pmatrix}
\begin{pmatrix} V^{-1} & & \\ & I & \\ & & \II \end{pmatrix}.
\end{align*}
Define $M := A_bV$ and $N := B_bV$ and set $j$ to be the largest integer multiple of $b$ that satisfies the inequalities:
\begin{equation}\label{eq:SDB_AED_deflationcriterion}
\norm{M_{1:b,1:j}}_F^2 + \abs{\lambda_{\rm max}(\Lambda_{1:j,1:j})}^2\norm{N_{1:b,1:j}}_F^2 \le {\tt tol} \left( \abs{\lambda_{\rm max}(\Lambda_{1:j,1:j})} + \abs{\lambda_{\rm min}(\Lambda_{1:j,1:j})} \right) \left( \abs{\lambda_{\rm max}(\Lambda_{1:j,1:j})} + \abs{\lambda_{\rm max}(\Lambda)} \right),
\end{equation}
for some tolerance ${\tt tol} > 0$. For eigenvalues only, ${\tt tol} = \epsilon$ and for both eigenvalues and eigenfunctions, ${\tt tol = \epsilon^2}$. If there is no integer $j$ that satisfies the inequalities, then double $k$ and reiterate until deflation sets in.

After determining that $j$ eigenvalues are to be deflated, we have a symmetric-definite operator pencil, and the second operator has unit diagonal elements for the first $k$ entries. We perturb the operators by setting $M_{1:b,1:j} \leftarrow 0$ and $N_{1:b,1:j} \leftarrow 0$ in order to deflate the first $j$ generalized eigenvalues. Thus, the remaining parts of the congruent operators are of the same form:
\begin{equation}
\arraycolsep=1.4pt
\begin{pmatrix} \mydiagdown & & & & & & \square & &\\
& \mydiagdown & & & & & \square & &\\
& & \mydiagdown & & & & \square & &\\
& & & \mydiagdown & & & \square & &\\
& & & & \myddots & & \myvdots & &\\
& & & & & \mydiagdown & \square & &\\
\square & \square & \square & \square & \mycdots & \square & \square & \lltriangle &\\
& & & & & & \urtriangle & \square & \myddots\\
& & & & & & & \myddots & \myddots\\
\end{pmatrix}~,\qquad
\begin{pmatrix} \mydiagdown & & & & & & \square & &\\
& \mydiagdown & & & & & \square & &\\
& & \mydiagdown & & & & \square & &\\
& & & \mydiagdown & & & \square & &\\
& & & & \myddots & & \myvdots & &\\
& & & & & \mydiagdown & \square & &\\
\square & \square & \square & \square & \mycdots & \square & \square & \lltriangle &\\
& & & & & & \urtriangle & \square & \myddots\\
& & & & & & & \myddots & \myddots\\
\end{pmatrix}
\end{equation}
We now perform congruence transformations to return the operator pencil to symmetric-definite banded form.

The following symmetric positive-definite matrix has a sparse Cholesky factorization:
\begin{equation}
\begin{pmatrix} I & N^\top\\
N & B_{kb}\\
\end{pmatrix}
= \begin{pmatrix} I\\
N & F\\
\end{pmatrix} \begin{pmatrix} I\\
N & F\\
\end{pmatrix}^\top, \quad{\rm where}\quad F F^\top = B_{kb} - N N^\top.
\end{equation}
Embedding the Cholesky factors in the identity operator:
\begin{equation}
\begin{pmatrix} I & N^\top & 0\\
N & B_{kb} & B_{2b}^\top \\
0 & B_{2b} & \BB_{k+b+1:\infty}
\end{pmatrix}
= \begin{pmatrix} I\\
N & F\\ & & \II
\end{pmatrix}
\begin{pmatrix} I & &\\
& I & F^{-1}B_{2b}^\top \\
& B_{2b}F^{-\top} & \BB_{k+b+1:\infty}
\end{pmatrix}
\begin{pmatrix} I\\
N & F\\ & & \II
\end{pmatrix}^\top.
\end{equation}
Similarly:
\begin{align*}
& \begin{pmatrix} \Lambda & M^\top & 0\\
M & A_{kb} & A_{2b}^\top \\
0 & A_{2b} & \AA_{k+b+1:\infty}
\end{pmatrix}\\
& = \begin{pmatrix} I\\
N & F\\ & & \II
\end{pmatrix}
\begin{pmatrix} \Lambda & (M-N\Lambda)^\top F^{-\top} &\\
F^{-1}(M-N\Lambda) & F^{-1}(A_{kb}-N\Lambda N^\top-MN^\top-NM^\top)F^{-\top} & F^{-1}A_{2b}^\top \\
& A_{2b}F^{-\top} & \AA_{k+b+1:\infty}
\end{pmatrix}
\begin{pmatrix} I\\
N & F\\ & & \II
\end{pmatrix}^\top.
\end{align*}
This congruence transformation results in the operator pencil:
\begin{equation}
\arraycolsep=1.4pt
\begin{pmatrix} \mydiagdown & & & & & & \blacksquare & &\\
& \mydiagdown & & & & & \blacksquare & &\\
& & \mydiagdown & & & & \blacksquare & &\\
& & & \mydiagdown & & & \blacksquare & &\\
& & & & \myddots & & \myvdots & &\\
& & & & & \mydiagdown & \blacksquare & &\\
\blacksquare & \blacksquare & \blacksquare & \blacksquare & \mycdots & \blacksquare & \blacksquare & \llblacktriangle &\\
& & & & & & \urblacktriangle & \square & \myddots\\
& & & & & & & \myddots & \myddots\\
\end{pmatrix}~,\qquad
\begin{pmatrix} \mydiagdown & & & & & & & &\\
& \mydiagdown & & & & & & &\\
& & \mydiagdown & & & & & &\\
& & & \mydiagdown & & & & &\\
& & & & \myddots & & & &\\
& & & & & \mydiagdown & & &\\
& & & & & & \blacksquare & \llblacktriangle &\\
& & & & & & \urblacktriangle & \square & \myddots\\
& & & & & & & \myddots & \myddots\\
\end{pmatrix}
\end{equation}
As in Algorithm~\ref{algorithm:SB_AED}, black symbols represent the cumulative modification of nonzero submatrices. The first operator is still symmetric and the second operator also retains symmetry and positive-definiteness.

Finally, we are free to use the same orthogonal similarity transformations described in Algorithm~\ref{algorithm:SB_AED}; these similarity transformations have no effect on the second operator since the first $\ell$ diagonal entries are unit, guaranteeing our return to symmetric-definite banded form:
\begin{equation}
\arraycolsep=1.4pt
\begin{pmatrix} \blacksquare & \llblacktriangle & & & & & & &\\
\urblacktriangle & \blacksquare & \llblacktriangle & & & & & &\\
& \urblacktriangle & \blacksquare & \llblacktriangle & & & & &\\
& & \urblacktriangle  & \blacksquare & \myddots & & & &\\
& & & \myddots & \myddots & \llblacktriangle & & &\\
& & & & \urblacktriangle & \blacksquare & \llblacktriangle & &\\
& & & & & \urblacktriangle & \blacksquare & \llblacktriangle &\\
& & & & & & \urblacktriangle & \square & \myddots\\
& & & & & & & \myddots & \myddots\\
\end{pmatrix}~,\qquad
\begin{pmatrix} \mydiagdown & & & & & & & &\\
& \mydiagdown & & & & & & &\\
& & \mydiagdown & & & & & &\\
& & & \mydiagdown & & & & &\\
& & & & \myddots & & & &\\
& & & & & \mydiagdown & & &\\
& & & & & & \blacksquare & \llblacktriangle &\\
& & & & & & \urblacktriangle & \square & \myddots\\
& & & & & & & \myddots & \myddots\\
\end{pmatrix}
\end{equation}
\end{algorithm}

\subsection{Summary of the data structures}\label{subsection:datastructures}

The adaptive spectral decomposition returns eigenfunctions as a product of orthogonal operators:
\begin{equation}
\LL = \QQ_1 \QQ_2 \cdots \QQ_\infty \Lambda \QQ_\infty^\top \cdots \QQ_2^\top \QQ_1^\top.
\end{equation}
The orthogonal operators $\QQ_i$ consist of the orthogonal matrices that implement the similarity transformations embedded in identity operators.

If the first $n$ eigenfunctions are desired, then the product of the orthogonal operators up to and beyond convergence of the eigenvalues is formed, at a cost of $\OO(n^3)$. If only a subset of the eigenfunctions is desired, then it is possible to reconstruct them in lower complexity. If the operator is applied to a function, then this may be performed in $\OO(n^2)$ operations.

In the same spirit, the adaptive generalized spectral decomposition returns generalized eigenfunctions as a product of operators:
\begin{equation}
\VV_\infty^\top \cdots \VV_2^\top \VV_1^\top \AA \VV_1 \VV_2 \cdots \VV_\infty = \Lambda, \quad{\rm where}\quad \VV_\infty^\top \cdots \VV_2^\top \VV_1^\top \BB \VV_1 \VV_2 \cdots \VV_\infty = \II.
\end{equation}

In certain cases, the complexity of our adaptive spectral decomposition may be linear. Thus, if we only need to {\em apply} the operator (or some function of it) to a function, then the product form returned by the adaptive spectral decomposition is ideal. We have observed linear complexity for a bounded perturbation of an (unbounded) spectral operator defining the expansion basis. In this setting, the $n^{\rm th}$ eigenfunction requires at most $n + \OO(\log\epsilon^{-1})$ degrees of freedom to resolve. This is reasonable from an approximation-theoretic point of view since the asymptotic distribution of roots of the sought eigenfunctions is the same as the spectral basis used to represent them. Precise bounds for the special cases of the Mathieu and spheroidal wave equations have been described by Volkmer~\cite{Volkmer-20-39-04}.

\section{The fractional Schr\"odinger equation}

There are many applications that benefit from the spectral decomposition of a linear differential operator. Consider the time-dependent fractional Schr\"odinger equation~\cite{Laskin-268-298-00,Laskin-66-056108-1-02}:
\begin{equation}\label{eq:FractionalSchrodinger}
\ii u_t = \HH_\alpha u,\quad{\rm where}\quad u(\theta,t=0) = u_0(\theta),
\end{equation}
and where $\HH_\alpha$ is a self-adjoint time-independent Hamiltonian and the L\'evy index $0<\alpha\le2$ denotes the strength of the implied fractional derivative. The formal solution of this equation is:
\begin{equation}
u(\theta,t) = \exp(-\ii t\HH_\alpha)u_0(\theta).
\end{equation}
If the partial spectral decomposition of $\HH_\alpha$ is available, then this problem no longer requires time-stepping. We consider $\HH_\alpha$ to be a fractional Mathieu operator~\cite{Mathieu-137-68,Kato-80} with periodic boundary conditions and a periodic potential on $[0,2\pi)$:
\begin{equation}
\HH_\alpha = \left(-\DD^2\right)^{\frac{\alpha}{2}} + 2q\cos(2\theta)\II.
\end{equation}
The real number $q$ determines the relative strength of the potential compared with the differential operator. We use a real-valued orthonormal Fourier series to represent the solution:
\begin{equation}
u(\theta,t) = u_0(t)\frac{1}{\sqrt{2\pi}} + u_{-1}(t)\dfrac{\sin(\theta)}{\sqrt{\pi}} + u_1(t)\dfrac{\cos(\theta)}{\sqrt{\pi}} + u_{-2}(t)\dfrac{\sin(2\theta)}{\sqrt{\pi}} + u_2(t)\dfrac{\cos(2\theta)}{\sqrt{\pi}} + \cdots +,
\end{equation}
where the ordering is chosen so that the regularity of the solution is observed in the decay of the coefficients. When $q=0$, the eigenfunctions are pure Fourier modes and otherwise, the fractional Mathieu operator is a bounded perturbation of the (unbounded) fractional Laplacian~\cite{Kwasnicki-20-7-17,Lischke-et-al-1801-09767}. Using the differentiation and linearization properties of trigonometric functions, the fractional Mathieu operator is:
\begin{equation}
\HH_\alpha = \begin{pmatrix} 0 & 0 & 0 & 0 & \sqrt{2}q\\ 0 & 1^\alpha-q & & & & q\\ 0 & & 1^\alpha+q & & & & q\\ 0 & & & 2^\alpha & & & & \ddots\\ \sqrt{2}q & & & & 2^\alpha\\ & q & & & & 3^\alpha\\ & & q & & & & 3^\alpha\\ & & & \ddots & & & & \ddots\\ \end{pmatrix}.
\end{equation}
Figure~\ref{fig:FractionalSchrodinger} shows the evolution of the fractional Schr\"odinger equation for various values of $\alpha$ and $q$. For an initial condition whose frequency is on the order of $q$, the particle densities of the classical Schr\"odinger equation barely notice the potential. On the other hand, for large deviations from a uniform potential, the fractional Schr\"odinger equation appears to focus the particle densities around the wells when $\alpha\ll2$.

\begin{figure}
\begin{center}
\includegraphics[width=\textwidth]{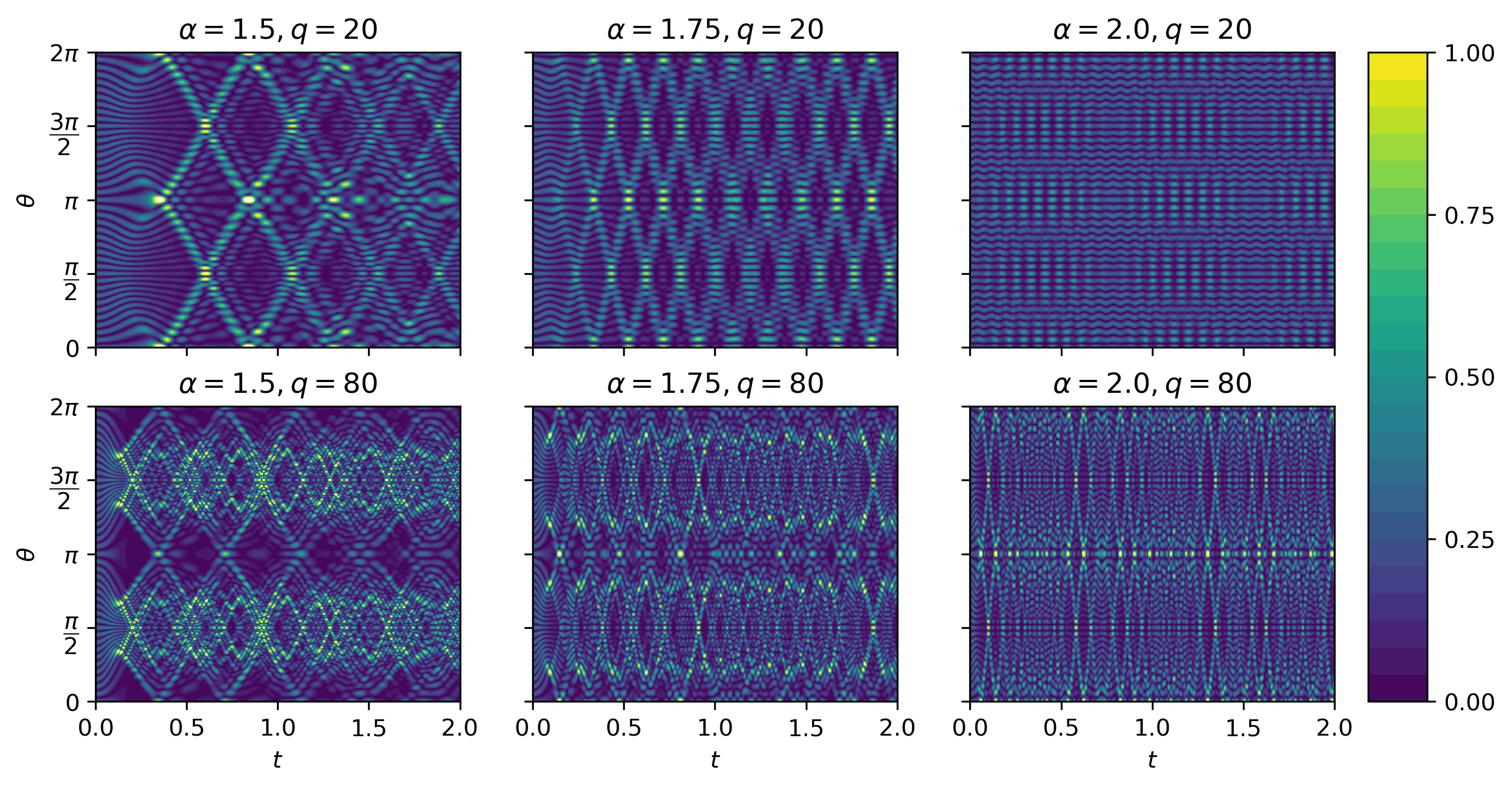}
\caption{Numerical solution of the fractional Schr\"odinger equation~\eqref{eq:FractionalSchrodinger} with two values of $q$ and three values of $\alpha$. The plots illustrate $\abs{u(\theta,t)}^2$ for the initial condition $u_0(\theta) = \cos(20\theta)/\sqrt{\pi}$. Although each plot has a different absolute maximum, we use the same colour scale to facilitate comparison.}
\label{fig:FractionalSchrodinger}
\end{center}
\end{figure}

It is straightforward to confirm that the fractional Schr\"odinger operator is unitary; therefore, the $L^2$-norm of the solution is a constant of motion. While unitarity preserving time-stepping schemes are of great interest for such problems, we note that unitarity preservation is also a property of a partial spectral decomposition. In Figure~\ref{fig:FractionalSchrodingerErrTime}, we monitor the numerical error in the $L^2$-norm of the Fourier coefficients starting with two $2$-normalized initial conditions: first, we start with $n = 10,001$ coefficients that mimic an analytic function $u_{\pm i} = \exp(-36(i+1)/\lceil\frac{n}{2}\rceil)$; second, with pseudorandom standard normal coefficients $u_{\pm i} = \NN(0,1)$, for $i=0,\ldots,\lfloor\frac{n}{2}\rfloor$. In both cases, the norm of the solution is very well preserved on long time scales. We also show the calculation time for the partial spectral decomposition and execution time of the eigenmatrix-vector product as a function of the truncation degree of the Fourier series, confirming the $\OO(n)$ complexity predicted in \S~\ref{subsection:datastructures}.

\begin{figure}
\begin{center}
\begin{tabular}{cc}
\hspace*{-0.2cm}\includegraphics[width=0.53\textwidth]{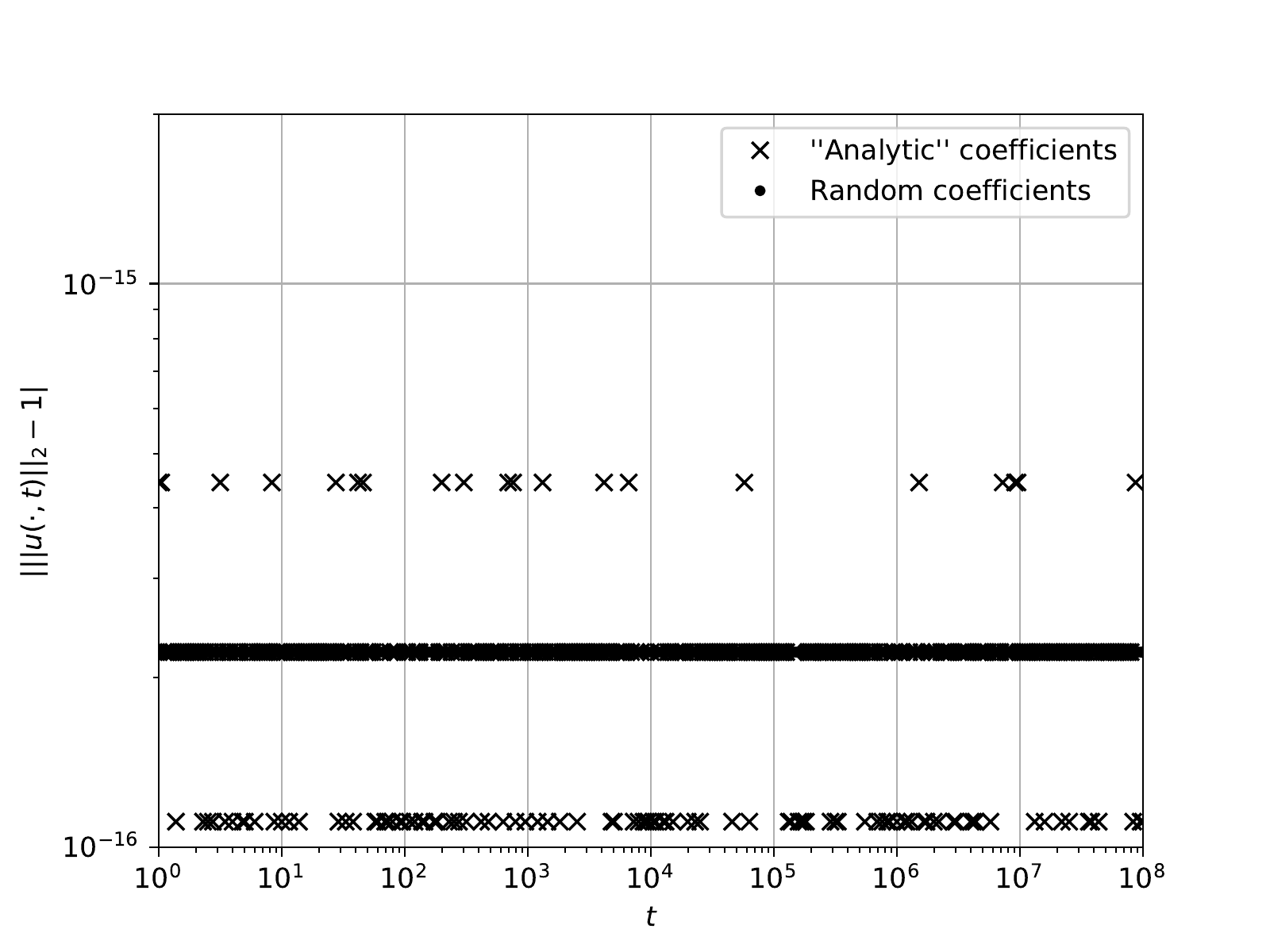}&
\hspace*{-0.65cm}\includegraphics[width=0.53\textwidth]{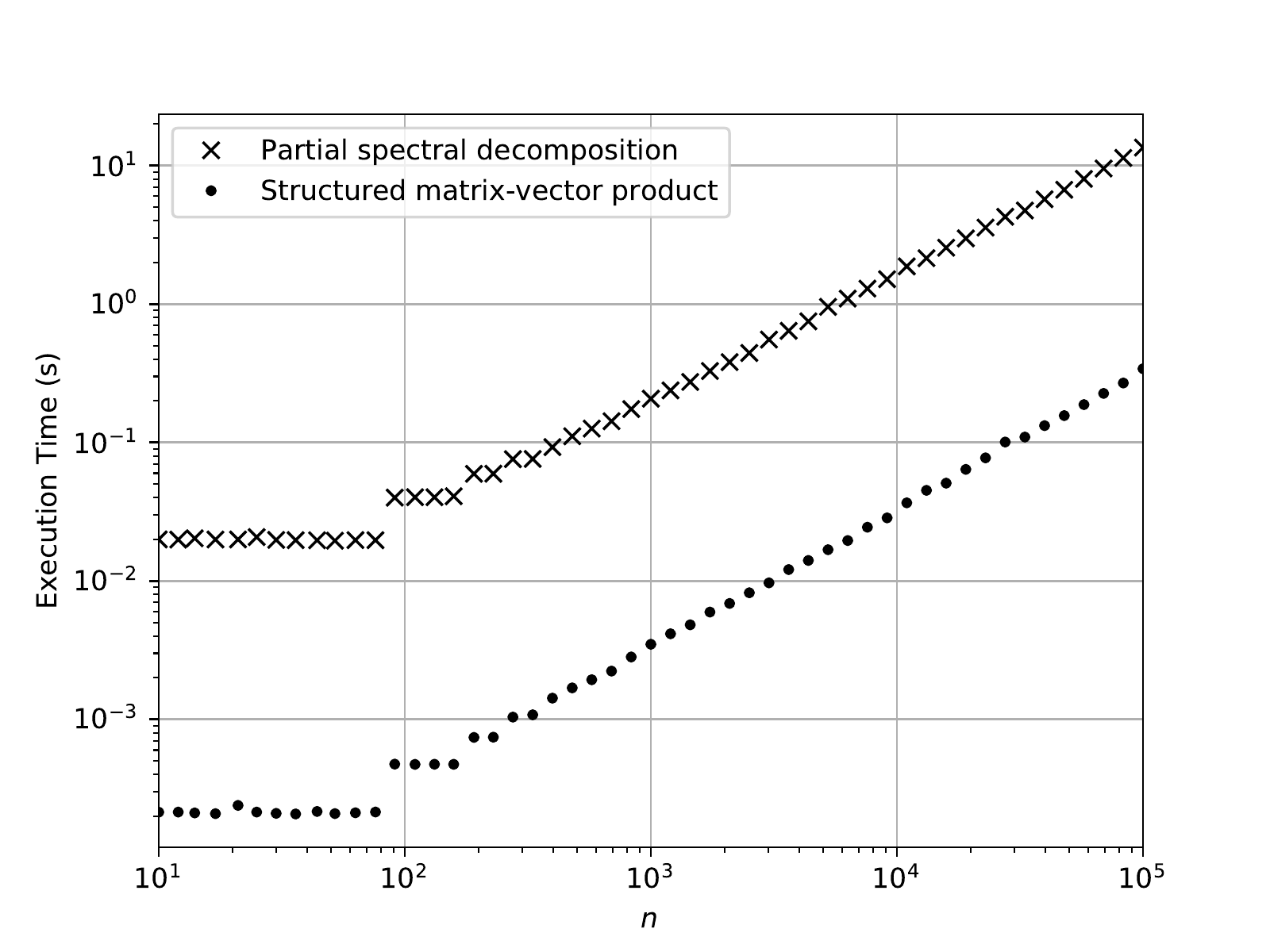}\\
\end{tabular}
\caption{Left: numerical error in the $L^2$-norm constant of motion for $n=10,001$ structured Fourier coefficients. Right: execution time of $\exp(-\ii \HH_\alpha) u_0(\theta)$ as a function of the truncation degree $n$ of $u_0$. In both plots, $\alpha = 1.5$ and $q = 80$.}
\label{fig:FractionalSchrodingerErrTime}
\end{center}
\end{figure}

\section{Conclusion \& outlook}

There are many avenues for future work. It is straightforward to consider symmetrizing eigenvalue problems from singular integral equations~\cite{Slevinsky-Olver-332-290-17}. Another natural extension is to eigenvalue problems in multiple dimensions. This includes tensor-product spectral methods~\cite{Townsend-Olver-299-106-15} and total-degree spectral methods on spheres, disks, and triangles~\cite{Olver-Townsend-Vasil-1902-04863}. In the former, operators must have a type of separable structure to permit fast linear algebra such as an eigen-analogue of the generalized Sylvester matrix equation, whereas in the latter, the logic of Algorithms~\ref{algorithm:SB_AED} and~\ref{algorithm:SDB_AED} is the same but the blocks are now banded matrices themselves whose dimensions increase with the block indices.

We presented an algorithm for an adaptive spectral decomposition where approximations are made in the deflation criterion and in the use of floating-point arithmetic. The FMM has been adapted to provide asymptotically fast approximate application of the eigenvectors of symmetric arrowhead matrices~\cite{Borges-Gragg-11-93,Gu-Eisenstat-16-172-95}, where the approximate precision is a tuneable parameter. The FMM may in principle be utilized to accelerate the spectral decomposition of the block arrow matrices that arise in Eq.~\eqref{eq:firstpartition}, though locating the spectrum may be more challenging in the blocked generalization.

Other variants of infinite-dimensional linear algebra could be pursued. These include an adaptive singular value decomposition and an adaptive Schur decomposition for non-self-adjoint problems. The symmetrizer may also be useful in the design and analysis of contours for spectral projectors in, e.g. the {\sf contFEAST} algorithm~\cite{Horning-Townsend-1901-04533}.

\section*{Acknowledgments}

We thank Yuji Nakatsukasa for a private communication on the spectral analysis of the model problem in \S~\ref{subsection:spectralanalysis}. We thank Andrew Horning and Marcus Webb for a careful reading of the manuscript. JLA is supported by ``la Caixa'' Banking Foundation, through the Postdoctoral Junior Leader Fellowship Programme. RMS is supported by the Natural Sciences and Engineering Research Council of Canada, through a Discovery Grant (RGPIN-2017-05514).

\bibliography{/Users/Mikael/Bibliography/Mik}

\end{document}

%% file: header.tex
\usepackage[T1]{fontenc}
\usepackage{times}
\usepackage{epsfig,graphicx,tabularx,color}
\usepackage[cp850]{inputenc}
\usepackage[english]{babel}
\usepackage{cite}
\usepackage{amsmath,amssymb,amsfonts,amsthm,mathrsfs,comment,mathdots,multirow,tikz}
\usepackage{rotating, longtable, relsize}
\usepackage[f]{esvect}
\usepackage{times}
\usepackage{fancyhdr}
\usepackage{lipsum}
\usepackage{nicefrac}

\usetikzlibrary{calc,3d}

\newtheorem{theorem}{Theorem}[section]

\newtheorem{algorithm}[theorem]{Algorithm}
\newtheorem{definition}[theorem]{Definition}
\newtheorem{proposition}[theorem]{Proposition}

\theoremstyle{definition}
\newtheorem{remark}[theorem]{Remark}

\def\sphline{\noalign{\vskip3pt}\hline\noalign{\vskip3pt}}

\def\XXint#1#2#3{{\setbox0=\hbox{$#1{#2#3}{\int}$}
     \vcenter{\hbox{$#2#3$}}\kern-.5\wd0}}

\def\ud{{\rm\,d}}

\def\C{\mathbb{C}}

\def\N{\mathbb{N}}
\def\P{\mathbb{P}}

\def\R{\mathbb{R}}

\def\AA{\mathcal{A}}
\def\BB{\mathcal{B}}
\def\CC{\mathcal{C}}
\def\DD{\mathcal{D}}

\def\HH{\mathcal{H}}
\def\II{\mathcal{I}}

\def\LL{\mathcal{L}}
\def\MM{\mathcal{M}}
\def\NN{\mathcal{N}}
\def\OO{\mathcal{O}}

\def\QQ{\mathcal{Q}}

\def\VV{\mathcal{V}}

\def\pr(#1){\left({#1}\right)}
\def\br[#1]{\left[{#1}\right]}

\def\abs#1{\left|{#1}\right|}
\def\norm#1{\left\|{#1}\right\|}

\def\ii{{\rm i}}
\def\for{\hbox{ for }}

\newcommand{\tril}{\operatorname{tril}}

\newcommand{\bs}[1]{\boldsymbol{#1}}

\DeclareRobustCommand{\rchi}{{\mathpalette\irchi\relax}}
\newcommand{\irchi}[2]{\raisebox{\depth}{$#1\chi$}} 